\documentclass{cmslatex}

\usepackage{latexsym, amssymb, enumerate, amsmath}

\renewcommand{\theequation}{\arabic{section}.\arabic{equation}}

\input{cf_macros.sty}

\input{datainverror}

\begin{document} 

\title{Numerical methods for multiscale inverse problems
\thanks{{Received date / Revised version date}
}}

\author{Christina Frederick
\thanks {School of Mathematics, Georgia Institute of Technology, Atlanta, GA, 30332, USA (cfrederick6@math.gatech.edu).}
\and Bj\"orn Engquist \thanks {Institute for
Engineering and Scientific Computing (ICES) and Department of Mathematics at The University of Texas at Austin,  Austin, TX 78712, USA (engquist@ices.utexas.edu).}}

\maketitle

\begin{abstract}
We consider the inverse problem of determining the highly oscillatory coefficient $\aeps$ in partial differential equations of the form $-\nabla \cdot \left(\aeps\nabla \ueps \right)+b\ueps = f$ from given measurements of the solutions. Here, $\epsilon$ indicates the smallest characteristic wavelength in the problem ($0<\epsilon\ll1$). In addition to the general difficulty of finding an inverse is the challenge of multiscale modeling, which is hard even for forward computations. The inverse problem in its full generality is typically ill-posed, and one common approach is to reduce the dimension by seeking effective parameters.  We will here include microscale features directly in the inverse problem and avoid ill-posedness by assuming that the microscale can be accurately represented by a low-dimensional parametrization. The basis for our inversion will be a coupling of the parametrization to analytic homogenization or a coupling to efficient multiscale numerical methods when analytic homogenization is not available. We will analyze the reduced problem, $b = 0$, by proving uniqueness of the inverse in certain problem classes and by numerical examples and also include numerical model examples for medical imaging, $b > 0$, and exploration seismology, $b < 0$.
\end{abstract}
\section{Introduction}

Multiscale modeling plays a crucial role in the development of mathematical and numerical methods for solving inverse problems arising in science and engineering. The design of accurate models must account for the numerous challenges involved in capturing a wide spectrum of time and spatial scales. Full resolution forward models come at a high computational cost, and many model-reduction techniques create difficulties in the mathematical formulation of the inverse problem. It is important to understand microstructure inversion problems where these challenges can be addressed by employing multiscale forward solvers and including prior knowledge in the inversion process. 

We will consider the problem of determining an unknown parameter in a forward model $G:X\rightarrow P$ from observational data. Here, $X$ and $P$ are function spaces, and the map $G$ is a solution operator for a partial differential equation of the form $\fwd(\aeps)=\ueps\in P$. The unknown parameter $\aeps\in X$ is a coefficient in the equation. The multiscale nature of the problem is indicated by the superscript $\epsilon$, where $\epsilon$ is the ratio of scales in the model ($0<\epsilon\ll1$). 

The collected measurements, denoted by $\data\in \RR^{n}$, are in practice obtained from experiments or electrical techniques. The mapping $\obs: X\times P\rightarrow \RR^{n}$ from the unknown parameter to the data, called the { observation operator}, is derived from the forward model. The solution to the inverse problem is then obtained by matching observations $\data = \obs(\aeps, \ueps)$, with predictions of the form $z=\obs(\widehat{a}, \widehat{u})$. In practice, the mapping from the parameter space to the space of predictions may differ from the observation operator, however here it is assumed that the mappings are the same.

Data-driven optimization problems require many simulations of the forward model, and when faced with balancing computational cost with accuracy, most approaches only deal with scientific models of large scale behavior and, for example in \cite{Nolen}, account for microscopic processes by using effective or homogenized equations to simplify computations. Homogenization theory \cite{Bensoussan, Jikov2011} provides the form of a reduced model that describes the effective behavior of the family of solutions $\{\ueps\}_{\epsilon>0}$; under certain ellipticity conditions, it is known that $\ueps\rightharpoonup \ueff$, as $\epsilon\to 0$, where $\ueff$ is the solution to an equation of the form $G(\aeff)=\ueff$, and the expression for the homogenized coefficient $\aeff$ is given by the theory.  

Ideas from homogenization theory can be used to account for the mismatch in scales between an effective model and the data generated by the full model, as demonstrated in \cite{Nolen}. It is shown that if only effective parameters in the forward model are desired, inversion can be performed using a macroscopic model for predictions.  In particular, the result is shown for cases where unknown coefficient $A$ is dependent on a single parameter $\theta\in \RR$ through a linear mapping $\theta\rightarrow \aeff(\theta)$. The inverse problem is formulated as a minimization problem,
\begin{align}
\underset{\theta\in \RR, \aeff=\aeff(\theta)}{\text{minimize     }}\quad \|\obs(\aeff, G(\aeff))-y^{\epsilon}\|.
\tag{$\text{IP}^{0}$}\label{effinv} \end{align} 
\noindent The the reduced formulation (\ref{effinv}) is often well-posed and results in a lower sensitivity to noise. A drawback of this approach is the loss of details about microscale features. 

In the current approach, full inversion is performed using effective forward models that are based on ideas from homogenization theory, as in \cite{Nolen}. We make use of the prior assumption of a {\it microscale} parametrization $m \rightarrow \aeps(m)$, where the  parameter $m$ is a scalar function depending on a low dimensional vector $ \theta\in\RR^{N}$, 
\begin{align}
m(x) = \sum_{i=1}^{N} \theta_{i} \psi_{i}(x), \labeleq{mpw}
\end{align}
where the functions $\psi_{i}$ are smooth functions defined on the interval $[\frac{i-1}{N}, \frac{i}{N})$, $1\leq i \leq N$. Macroscopic predictions are made using ideas from homogenization theory, gaining the benefits of the previous approach, and the corresponding minimization problem is
\begin{align}
\underset{m, \aeps=\aeps(m)}{\text{minimize }}\quad \|\obs(\aeps, G(A))-y^{\epsilon}\|.\tag{$\text{IP}$}\label{microinv} \end{align} 
We will give sufficient conditions for uniqueness and boundary stability of solutions to a continuous inverse problem for elliptic partial differential equations that is related to (\ref{microinv}). These conditions correspond to a classification of certain physical features of the microstructure that are preserved under homogenization.

The following is a list of main strategies for solving inverse problems involving multiscale model parameters. 

\begin{enumerate}[I.]
\item {\bf Full coefficient inversion.} {Full coefficient inversion is performed by minimizing the distance between model predictions and the given data,
\begin{align}
\underset{\aeps}{\text{   minimize }}\quad \|\mathscr{G}(\aeps, G(\aeps))-y^{\epsilon}\|.\tag{$\text{IP}^{\epsilon}$}\label{fullinv} 
\end{align} Determining the original coefficient using high resolution predictions comes at a large computational cost and is often ill-posed due to the presence of multiple local minima in the associated cost functionals. Therefore, we omit this case from our computations.}

\item \textbf{Indirect microscale parameter estimation.}\label{TWOSTAGEproblem} {An indirect method for solving (\ref{microinv}) involves a two-stage procedure. The first step is to solve the problem of estimating the parameter $\hat{\aeff}$ in the effective model that best matches the given data. The second step involves determining the microscale parameter $m$ such that the homogenized coefficient corresponding to $\aeps(m)$ is $\hat{A}$.} This method can be written as

\begin{align*}
1.  &\underset{\aeff}{\text{ minimize}}\quad\|\obs(\aeff, \fwd(\aeff))-y^{\epsilon}\| \rightarrow \hat{A}\\
2.  & \underset{m}{\text{ minimize}}\quad\|\aeff(m) - \hat{\aeff} \|.
\end{align*}

\item {\bf Direct microscale parameter estimation.}\\
In this case, (\ref{microinv}) is solved in one step, where predictions of the forward model are made using techniques from multiscale modeling and numerical homogenization. In our experiments we consider two methods.

\begin{enumerate}[a.]
\item  {\it Known homogenization}.  \label{EFFproblem.} If the explicit form of the homogenized coefficient $\aeff(m)$ corresponding to each parameterized coefficient $\aeps(m)$ is known, a macroscopic method can be used to solve the effective equation.

\begin{equation}
 \begin{aligned}
\underset{m\in X, \aeff=\aeff(m)}{\text{minimize }}\quad \|\obs(\aeff, \fwd(\aeff))-y^{\epsilon}\|.
  \end{aligned}
  \labeleq{effinv}
\end{equation}

\item \textit{HMM}.   Often, the explicit form of the homogenized coefficient is not available, preventing the direct computation of macroscopic predictions. This issue can be overcome numerically with the heterogeneous multiscale method, or HMM, introduced by E and Engquist \cite{E2003}. HMM provides a framework for the design of methods that capture macroscale properties of a system using microscale information. The inverse problem is formulated as

\begin{equation}
 \begin{aligned}
\underset{m, \aeps=\aeps(m)}{\text{minimize }}&\quad \|\obs(\aeps, \fwd(\aeff))-y^{\epsilon}\|.
  \end{aligned}
  \labeleq{hmminv}
\end{equation}
Here, the forward model $\fwd(\aeff)$ is approximated using methods for numerical homogenization of the predicted coefficient $\aeps$. In the experiments we use the heterogeneous multiscale method (HMM).

\end{enumerate}
\end{enumerate}

In \S\ref{sec:forwardmodels} we give a brief background on periodic homogenization and introduce key microstructure models that demonstrate the main ideas of this work. In \S \ref{sec:msinv}, a multiscale inverse problem related to (\ref{microinv}) is formulated in the classical setting of inverse problems for elliptic equations. Uniqueness and boundary stability results are given.  In \S \ref{sec:hmm} we describe the implementation of the finite element heterogeneous multiscale method. Numerical results for parameter inversion are provided in \S \ref{sec:numericalresults}. In  \S \ref{sec:qpat} and \S \ref{sec:helm} we consider model problems from applications in medical imaging and geophysics. Then we conclude in \S \ref{sec:conclusion}.


\subsection{Notation}

The averaging operator is denoted by $\<f \>_{X} = \frac{1}{|X|}\int_{X}f(y)dy$, where $|X|$ is the volume of the set $X\subset \RR^{d}$.
For most examples $Y=[0,1]^d$, and unless otherwise stated,  $\<\cdot\>=\<\cdot\>_Y$.
For any domain $D$, we use the Sobolev space $W^{m,p}(D)$ with Sobolev norm $\|\cdot\|_{W^{m,p}}$. If $D=\Omega$, we omit $D$. Moreover, if $D=\Omega$ and $p=2$, we denote by $H^{m}(\Omega)$ the Sobolev space $W^{m,2}(\Omega)$, the usual $L^{2}$ inner product by $(\cdot, \cdot)$ and the Sobolev norm by $\|\cdot\|_{m}$. The norm on the Banach space of bounded linear operators between $H^{1/2}(\partial\Omega)$ and $H^{-1/2}(\partial\Omega)$ is denoted by $\|\cdot\|_{*}$.

\section{Homogenized forward model}\label{sec:forwardmodels}
Let $\Omega\subset\RR^{d}$ be a bounded domain with $C^{2}$-boundary. We consider equations for which there is a well established homogenization theory \cite{Bensoussan, Jikov2011},
\begin{align} 
-\nabla\cdot(\aeps \nabla \ueps) +b\ueps = f & \text{ in }  \Omega, \label{eq:ms} \end{align}
where $f$ and $b$ are given bounded functions and $\aeps(x)=a(x,x/\epsilon)$ for a given matrix function $a$ that is locally periodic, symmetric, and uniformly positive definite.  

A constant positive definite matrix $\aeff$ is said to be the homogenized matrix for $\a$, if for any bounded domain $\Omega \subset \RR^{d}$ and any $f\in H^{-1}(\Omega)$ the solutions of the Dirichlet problem (\ref{eq:ms}) possess the following property of convergence: as $\epsilon\to 0$, $\ueps {\rightharpoonup}{\ueff} $  in $H^{1}_{0}(\Omega)$ and ${\aeps}\nabla \ueps {\rightharpoonup} \aeff\nabla \ueff$ in $L_{2}(\Omega)$, where $\ueff$ is the solution of the Dirichlet problem
\begin{align} 
-\div(\aeff(x) \nabla\ueff(x))+b\ueff  = f & \text{ in } \Omega.  \label{eq:eff}\end{align}
The homogenized matrix has a closed form expression,
\begin{align}
A(x) = \frac{1}{|Y|}\int_{Y}({a}(x,y) Id+{a}(x,y) \nabla_{y}\chi) dy,\labeleq{effcoeff}
\end{align}
\noindent where $\chi = (\chi_{1}, \chi_{2})$ solves the cell problems,
\begin{align}
-\nabla_{y} \cdot ({a}(x,y)  \nabla_{y}\chi )=\nabla_{y}\cdot {a}(x,y) Id,\labeleq{cellb}
\end{align}
with the constraint $\chi(x,y)$ is $Y-$periodic in $y$ and $\langle \chi(x,\cdot) \rangle = 0$. 
In general, $\refeq{effcoeff}$ must be calculated using solutions to cell problems. Explicit formulas are known in one dimension and also in certain higher dimensional models, such as those describing layered media \cite{Jikov2011}. Even if the original coefficients are isotropic, the process of homogenization introduces anisotropy.

\subsection{Microstructure models}

The ideas in the remaining sections can be understood in terms of the following examples of parametrized microstructures. Let $m\in L^{\infty}(\Omega)$ be a function taking values in the interval $I_{\lambda}=[\lambda^{-1},\lambda]$ for $\lambda>1$, and let
\begin{align}
\aeps(m(x), x) = a(m(x), x/\epsilon) \text{Id}, \labeleq{aeps-m}
\end{align} 
where $a(x,y)$ is smooth, bounded, and periodic in the second variable and Id denotes the $d\times d$ identity matrix.
The first three models below are commonly used in the analysis of layered materials (see Figure \ref{fig:microlayered}). The last two models represent properties of materials containing cell microstructures (see Figure \ref{fig:micro2}).
\begin{enumerate}[A.]
\item {\bf Amplitude.} \label{AMP} For a positive constant $a_0$ and a periodic, bounded function $\bar{a}(y) = \bar{a}(y_{2})$ with $\<\bar{a}\>=0$, the parametrization of the amplitude of oscillations is modeled by
\begin{align}
a(m, y)=a_0+m \bar{a}(y_{2}).& \label{eq:amplitude} 
\end{align}
\item {\bf Volume fraction.}\label{VF}  A special case of layered materials are two-phase laminates, where the parameter $m$ determines the volume fraction of each, 
\begin{align}
a(m, y)=\begin{cases} k_{1} & 0\leq  y_{2} < m\\
 k_{2}& m\leq y_{2} <1.\end{cases}\label{eq:af}
\end{align}

\item {\bf Angle.} \label{ANG} Here, $\hat{a}$ is a periodic function and $\sigma_{m}$ is a matrix of rotation,
\begin{align}
a(m,y)=\hat{a}(\sigma_{m} y), \qquad \sigma_{m} =\begin{pmatrix}\cos (2\pi m) &\sin (2\pi m) \\ -\sin (2\pi m)&\cos (2\pi m)\end{pmatrix} .  
 \label{eq:angle} 
\end{align}

\end{enumerate}

\begin{figure}
\includegraphics[width=.3\linewidth]{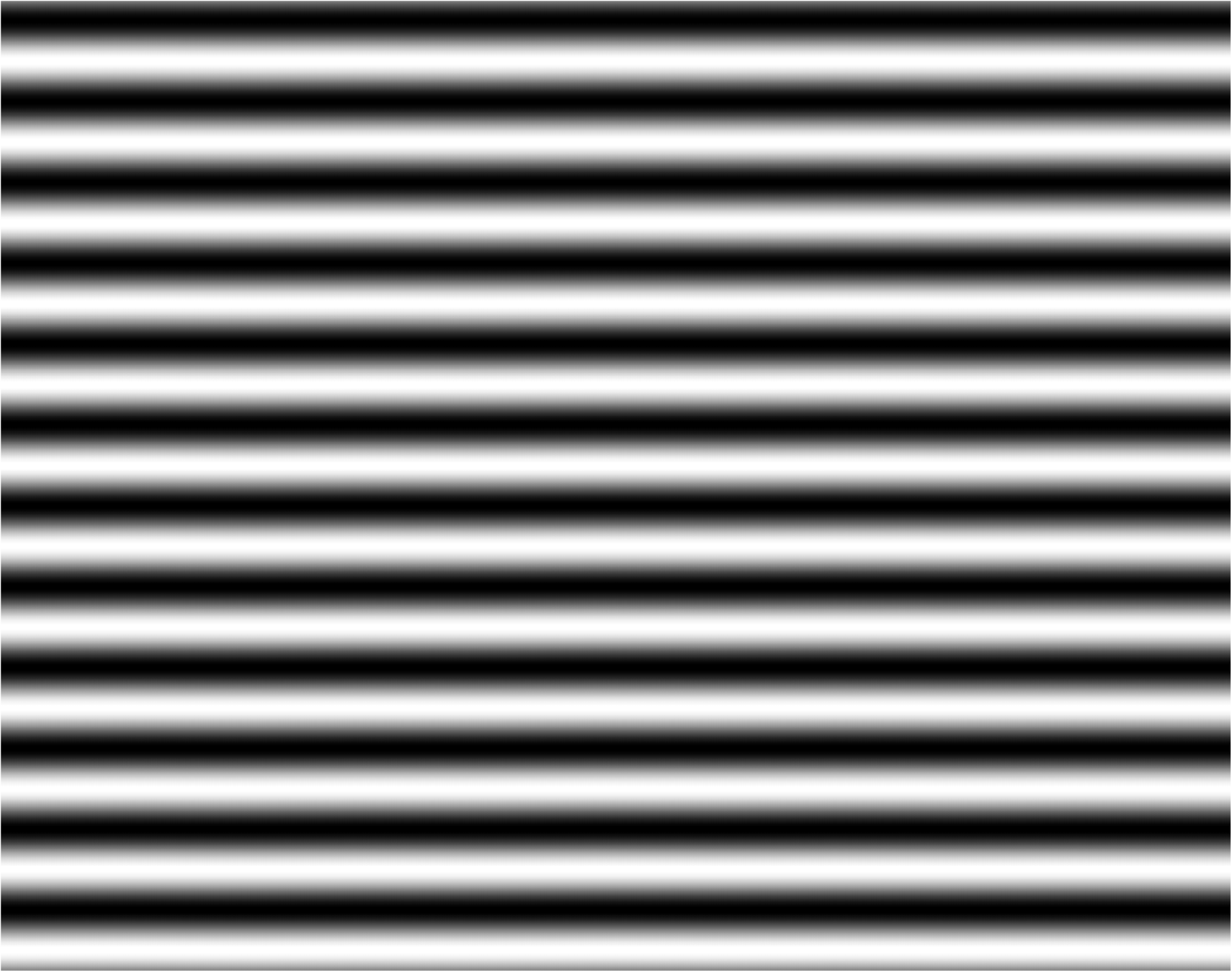}\hfill{\includegraphics[width=.3\linewidth]{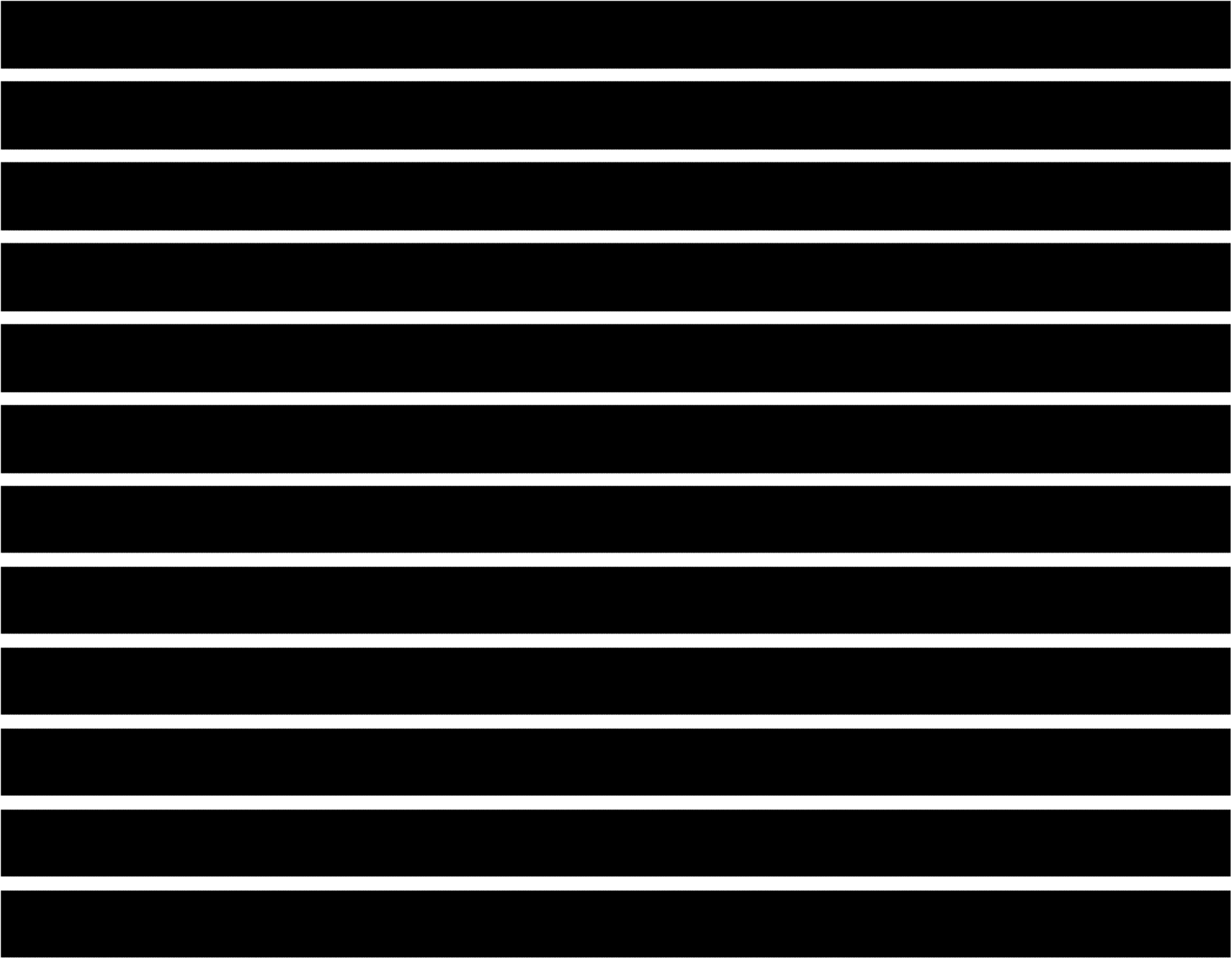}}\hfill
{\includegraphics[width=.3\linewidth]{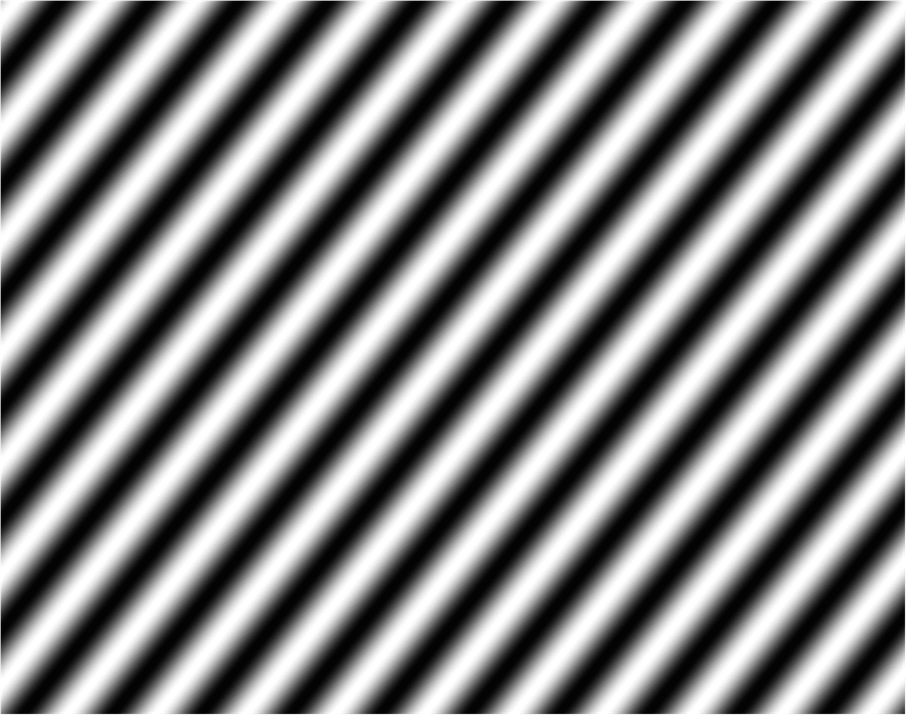}}
\caption{{\bf Layered microstructures.} From left to right: A - Amplitude, B -  Volume fraction, C - Angle.}
  \label{fig:microlayered}
\end{figure}

\newcommand{\boxY}{Y}
\begin{enumerate}[A.]\setcounter{enumi}{3}
\item\textbf{Amplitude in cell structures.}\label{AMP2d}  The analog of Model \ref{AMP} is a class of separable functions $a$,
\begin{align}
a(m, y)= {a}_1(m,y_1){a}_2(m,y_2),\label{eq:amp2d}
\end{align}
where $ a_1$ and $ a_2$ of the type in \refeq{amplitude}.

\item \textbf{Volume fraction in cell structures.}\label{VF2D} The analog of Model \ref{VF} is
\begin{align}
a(m, y)= \begin{cases} k_{1}  & y\in {m} \boxY\\
 k_{2}& \text{otherwise}, \end{cases}\label{eq:vf2d}
\end{align}
where $k_{1}$ and $k_{2}$ are positive constants.

\end{enumerate}

\begin{figure}
{\includegraphics[width=.45\linewidth]{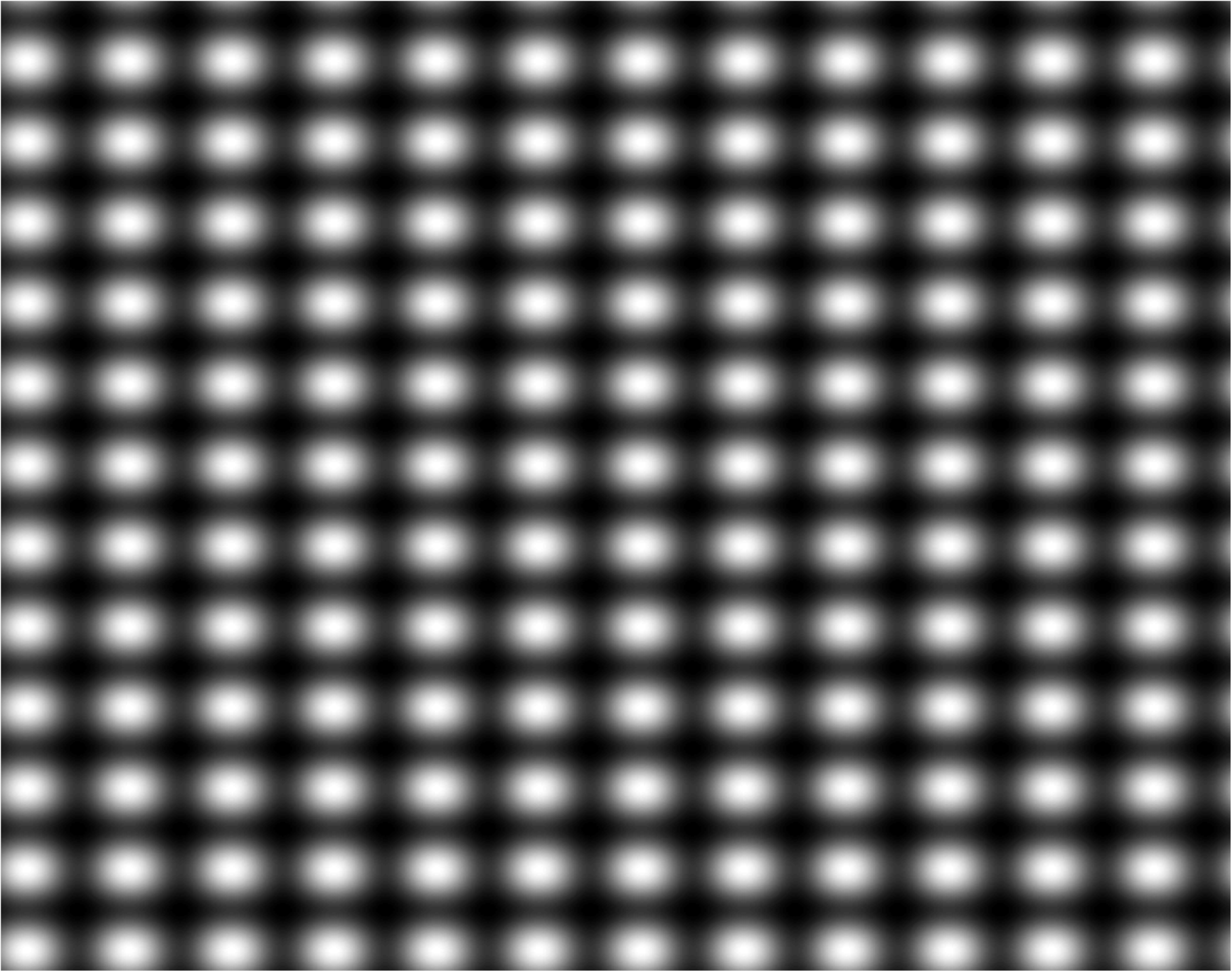}}\hfill
{\includegraphics[width=.45\linewidth]{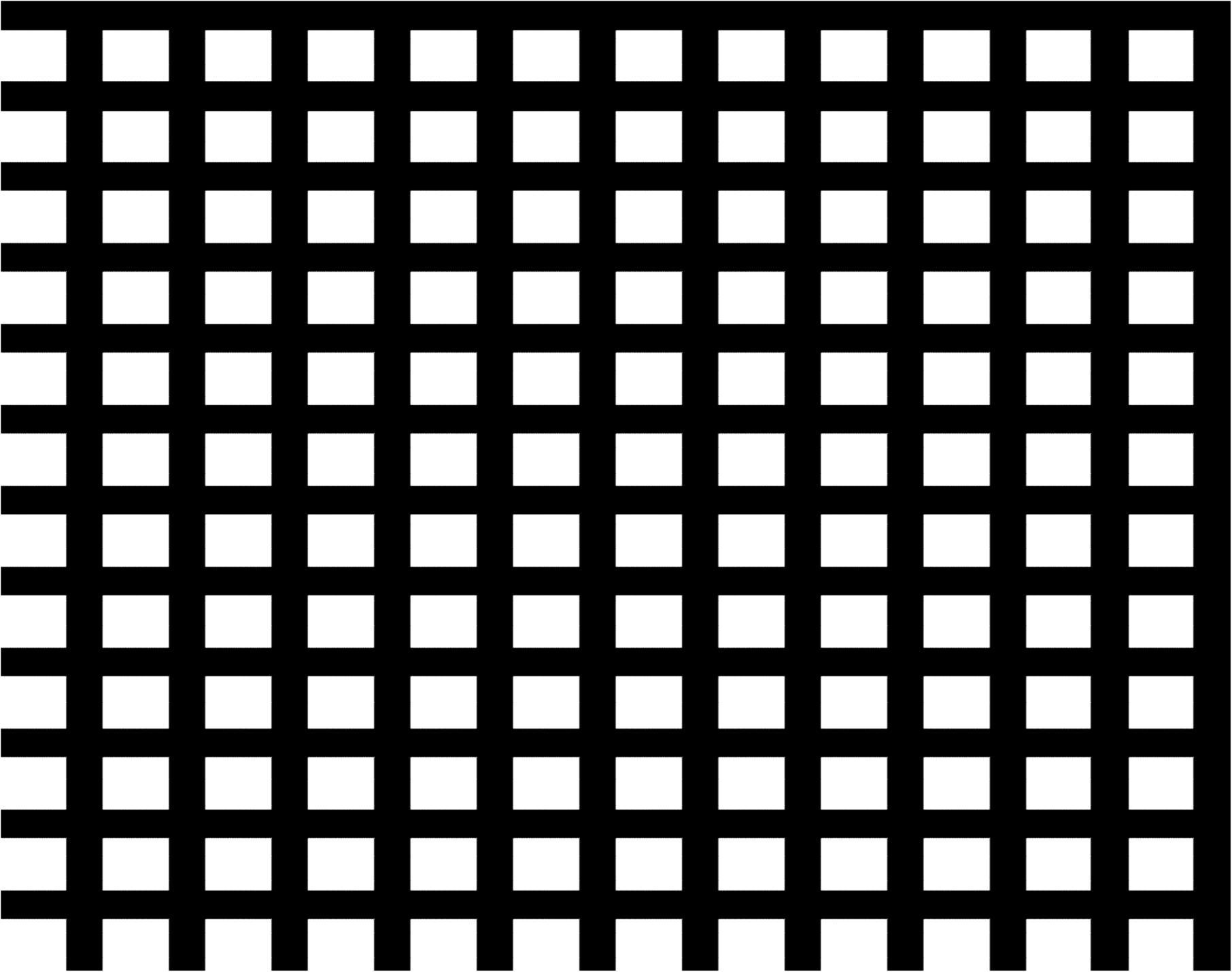}}\hfill
\caption{{\bf Cell microstructures.} Cell structures  D - Amplitude (left) and  E - Volume Fraction (right).}
  \label{fig:micro2}
\end{figure}

\subsubsection{Explicit calculations of homogenized microstructures}

The locally periodic microstructure models (\ref{AMP} $-$ \ref{VF2D}) admit homogenized matrix functions of the form $\aeff(m,\cdot) = A(m)$.
\noindent In two dimensions, calculations of \refeq{effcoeff} can be made explicit. 

\subsubsection*{\bf Layered materials (\ref{AMP} and \ref{VF})} 

The cell problems \refeq{cellb}  can be expressed as
\begin{align*}
-\frac{\partial}{\partial{y_{2}}}  ({a}(m, y_2) \frac{\partial}{\partial{y_{2}}}\chi_{1} )&=0,\\ 
-\frac{\partial}{\partial{y_{2}}}  ({a}(m, y_2) \frac{\partial}{\partial{y_{2}}}\chi_{2} )&= \frac{\partial}{\partial{y_{2}}} {a}(m, y_2),
\end{align*}
where the solutions are of the form $\chi=(\chi_{1}(m, y_2) , \chi_{2}(m, y_2) )$. Integration from $0$ to $y_{2}$ gives
\begin{align}
{a}(m, y_2) \frac{\partial \chi_{1}}{\partial{y_{2}}} &= c_{1},\labeleq{bdc1}\\
{a}(m, y_2) \frac{\partial \chi_{2}}{\partial{y_{2}}} &= -{a}(m, y_2) +d_{1},\labeleq{bdc2}
\end{align}
for some constants $c_{1}(m)$ and $d_{1}(m)$. Since ${a}$ is strictly positive, we can divide \refeq{bdc1} and \refeq{bdc2} by ${a}(m, y_2) $ and integrate from $0$ to $y_{2}$ again
\begin{align*}
 \chi_{1} &= c_{1}\int_{0}^{y_{2}}\frac{1}{{a}(m, \xi) }d\xi+c_{2},\qquad\chi_{2} = -y_{2}+d_{1}\int_{0}^{y_{2}}\frac{1}{{a}(m, \xi) }d\xi+d_{2}.
\end{align*}
Now, using periodicity, $\chi_{l}(0,m)=\chi_{l}(1,m)$ it follows  that $c_{1}=0,$ and $d_{1}=\<{{a}(m, \cdot) ^{-1}}\>^{-1}$.
Therefore \refeq{bdc1} and \refeq{bdc2} become
\begin{align*}
{a}(m, y_2) \frac{\partial \chi_{1}}{\partial{y_{2}}} &=  0,\\
{a}(m, y_2) \frac{\partial \chi_{2}}{\partial{y_{2}}} &= -{a}(m, y_2) +\<{{a}(m, \cdot) ^{-1}}\>^{-1}.
\end{align*}
Substituting these expressions into \refeq{effcoeff} results in the explicit form of the homogenized coefficient,
\begin{align}
\qquad &A(m)= \begin{pmatrix} \langle {a(m, \cdot)}\rangle &0 \\
0&\langle {a(m, \cdot)}^{-1}\rangle^{-1}\end{pmatrix}.\label{eq:effamp}
\end{align}

 \subsubsection*{\bf Materials with oriented layers (\ref{ANG})} 
Suppose $a$ is of the form $ \refeq{angle} $ with $\hat{a}(y)=\hat{a}(y_{2})$ for all $y=(y_{1}, y_{2})$ and  $\sigma=\sigma_{m}$. 
For a bounded set $\Omega\subset \RR^{2}$, consider the scalar problems
\begin{align*}
&\int_{\Omega}\nabla \psi \cdot \hat{a}(\x/\epsilon) \nabla u^{\epsilon} d\x= 0,  \quad  \forall \psi \in H_{0}^{1}(\Omega)  \quad \text{ for } u^{\epsilon}\in H_{0}^{1}(\Omega), \\
&\int_{\Omega}\nabla \psi \cdot {\hat{A}} \nabla \ueff d\x = 0,  \quad  \forall \psi \in H_{0}^{1}(\Omega)  \quad\text{ for }  \ueff\in H_{0}^{1}(\Omega).
\end{align*}

Now consider the change of variables $\x = \sigma\y$ where $\sigma$ is an orthogonal transformation from $\RR^{2}$ to $\RR^{2}$. We obtain the Dirichlet problems for $\Omega'=\sigma^{-1}\Omega$,
\begin{align*}
&\int_{\Omega'}\nabla_{y} \psi \cdot \hat{a}(\sigma\y/\epsilon) \nabla_{y} u^{\epsilon}(\sigma\y) d\y= 0,\\
&\int_{\Omega'}\nabla_{y} \psi \cdot \sigma \hat{A} \sigma ^{-1}\nabla_{y} \ueff(\sigma \y) d\y = 0.
\end{align*}

Since $\ueps(\sigma \y)\rightharpoonup \ueff(\sigma y)$ in $H^{1}_{0}(\Omega')$, it follows that the homogenized coefficient corresponding to  \refeq{angle}  is $\sigma \hat{A} \sigma^{-1}$, or
\begin{align}
\qquad &A(m)=\sigma_{m}^T\begin{pmatrix} \langle \hat{a}\rangle &0 \\
0&\langle \hat{a}^{-1}\rangle^{-1}\end{pmatrix}
 \sigma_{m}.\label{eq:effangle}
\end{align}

\subsubsection*{\bf Homogenization of cell structures (\ref{AMP2d} and \ref{VF2D})} 
It is well known \cite{Jikov2011, Pavliotis2008} that the homogenized coefficient corresponding a separable function of the form \refeq{amp2d} is the diagonal matrix
\begin{align}
A(m)=  \begin{pmatrix} \langle {a_1(m, \cdot)}^{-1}\rangle^{-1}\langle a_{2}(m, \cdot)\rangle  &0 \\
0& \langle {a_2(m, \cdot)}^{-1}\rangle^{-1}a_{1}(m, \cdot)\rangle\end{pmatrix}.\label{eq:effamp2d}
\end{align}

\noindent Matrix functions $\mathcal{A}(y) = a(m,y)$Id, where $a(m, y)$ has the form \refeq{vf2d}, can be derived explicitly. The solutions to the cell problems \refeq{cellb} corresponding to $a(m,y)$ of the type \refeq{vf2d} are equivalent to
\begin{align}
\frac{\partial}{\partial{y_{i}}}  \({a}(m, y)\frac{\partial\chi_{k}}{\partial{y_{i}}} \)&=-\frac{\partial}{\partial{y_{i}}} {a}(m, y),  &i=1, 2, \quad k=i,\labeleq{unk1}\\ 
\frac{\partial}{\partial{y_{i}}}  \({a}(m, y)\frac{\partial\chi_{k}}{\partial{y_{i}}} \)&=0,  &i=1, 2, \quad k\neq i. \nonumber
\end{align}
For $i=1$, integration from $0$ to $y_{1}$ gives \begin{align}
 {a}(m,y)\frac{\partial \chi_{1}}{\partial{y_{1}}} &= -{a}(m, y)+c_{1},\labeleq{unk2}\\
 {a}(m,y)\frac{\partial \chi_{2}}{\partial{y_{1}}} &=\tilde{c}_{1},\labeleq{unk3}
\end{align}
where $c_{1} = c_{1}(m, y_{2})$ and $\tilde{c}_{1}=\tilde{c}_{1}(m, y_{2})$. Since $a(m,y)$ is positive, we can divide by $a(m,y)$ and integrate from $0$ to $y_{1}$ again, giving
\begin{align*}
\chi_{1} &= -y_{1}+c_{1}\int_{0}^{y_{1}}\frac{1}{{a}(m, y)}dy_{1}+c_{2},\\
\chi_{2} &= \tilde{c}_{1}\int_{0}^{y_{1}}\frac{1}{{a}(m, y)}dy_{1}+\tilde{c_{2}},
\end{align*}
where $c_{2}$, $\tilde{c}_{2}$ are also  functions of only  $m$ and $y_{2}$. Applying the periodic boundary conditions  $\chi_{1}|_{y_{1}=0}=\chi_{1}|_{y_{1}=1}$ results in
\begin{align*}
c_{1}(m,y_{2})&=\({\int_{0}^{1}\frac{1}{{a}(m, y)}dy_{1}}\)^{-1},\\
\tilde{c}_{1}(m,y_{2}) &=0.
\end{align*}
Therefore \refeq{unk2} and \refeq{unk3} become
\begin{align*}
 {a}(m, y)\frac{\partial \chi_{1}}{\partial{y_{1}}} &= -{a}(m, y)+\({\int_{0}^{1}\frac{1}{{a}(m, y)}dy_{1}}\)^{-1},\\
{a}(m,y)\frac{\partial \chi_{2}}{\partial{y_{1}}} &=0.
\end{align*}
A similar argument applies to $i=2$, resulting in
\begin{align*}
{a}(x,y) \nabla_{y}\chi = \begin{pmatrix}-{a}(m, y)+\(\int_{0}^{1}\frac{1}{a(m, y)} dy_{1}\)^{-1} & 0 \\
0& -{a}(m, y)+\({\int_{0}^{1}\frac{1}{{a}(m, y)}dy_{2}}\)^{-1}\end{pmatrix}.
\end{align*}

Substituting this expression into \refeq{effcoeff} results in the closed form for the isotropic homogenized coefficient $A(m) = \bar{a}(m) \text{Id}$,
where 
\begin{align}
\bar{a}(m) = \int_{0}^{1}\({\int_{0}^{1}\frac{1}{{a}(m, y)}dy_{1}}\)^{-1}dy_{2} =  \frac{m k_{1}k_{2}}{m (k_{2}-k_{1})+k_{1}} + (1-m) k_{2}. \labeleq{unkeffAx}
\end{align}

\section{Multiscale analysis for inverse conductivity problems}\label{sec:msinv}

Let $\Omega$ be an open, bounded region in $\RR^d$, $d\geq 2$, that has a sufficiently smooth boundary $\partial \Omega$. The forward model studied in the classical theory of inverse problems is the Dirichlet problem,
\begin{align}
 -\div\(A\nabla u\) = 0 \text{ in } \Omega.\label{eq:eit}
\end{align}
The coefficient $A$ is in general a uniformly positive definite, symmetric, $d\times d$ matrix \cite{Sylvester1990a, Uhlmann}.

\begin{definition} For $g,h\in H^{1/2}(\partial\Omega)$ let $u\in H^{1}(\Omega)$ be the weak solution to  \refeq{eit} subject to  $u|_{\partial\Omega}=g $, and let $v$ be an arbitrary function in $H^{1}(\Omega)$ that satisfies $v|_{\partial\Omega}=h$. The Dirichlet-to-Neumann map $\Lambda_{A}:H^{1/2}(\partial\Omega)\rightarrow H^{-1/2}(\partial\Omega)$ is defined by
\begin{align*}
\langle\Lambda_{A}g, h\rangle = \int_{\Omega}A(x) \nabla u(x)\cdot \nabla v(x) dx.
\end{align*}
\end{definition}

The inverse boundary value problem of Calder\'{o}n \cite{Calderon} is to recover $A$ from knowledge of the \DtN map $\Lambda_{A}$. In general, the inverse problem is highly ill-posed. A main challenge is to prove the stability of the problem, that is, the continuous dependence of the unknown $A$ on the data $\Lambda_{A}$.

An approach that can be applied to anisotropic coefficients assumes the prior knowledge of a parametrization 
\begin{align}
m(x)\rightarrow A(m(x),x).\labeleq{macroparam}
\end{align}

\begin{definition}[Adapted from Definition 2.2 in \cite{Alessandrini2001}]Given $p>d$, $E>0$, and denoting by $\text{Sym}_{d}$ the class of $d\times d$ real-valued symmetric matrices, we say $A(\cdot, \cdot)\in \mc{H}$ if the following conditions are satisfied:
\begin{align*}
A&\in W^{1, p}([\lambda^{-1}, \lambda]\times\Omega, \text{Sym}_{d}), \\
D_{m}A&\in W^{1, p}([\lambda^{-1}, \lambda]\times\Omega),\\
&\textrm{supess}_{m\in\I}\(\|A(m,\cdot)\|_{L^{p}(\Omega)} +\|D_{x}A(m,\cdot)\|_{L^{p}(\Omega)},\right.\\
&\left. \qquad \qquad \qquad + \|D_{m}A(m,\cdot)\|_{L^{p}(\Omega)}+\|D_{m}D_{x}A(m,\cdot)\|_{L^{p}(\Omega)}  \)\leq E,\\
\lambda^{-1}|\xi|^{2}& \leq A(m,x)\xi \cdot \xi \leq \lambda |\xi|^{2} \text{ for a.e. $x\in\Omega$ and  all } m\in\I, \xi\in \RR^d.
\end{align*}
The essential supremum is denoted by \textit{supess}. In addition, the following monotonicity condition must also be satisfied:
\begin{align}
D_{m}A(m,x)\xi \cdot \xi \geq E^{-1} |\xi|^{2}\labeleq{monotonicity}
\end{align}
for a.e. $x\in\Omega$ \text{and all} $m\in\I, \xi\in \RR^d$.
\end{definition}
The following theorems, adapted to our context, are from \cite{Alessandrini2001}. The first is a boundary stability result and the second gives a global uniqueness result for matrices $A(\cdot,\cdot)\in\mc{H}$.

\begin{theorem}[\cite{Alessandrini2001}, Theorem 2.1]\label{thm:Aless-stability}
Given $p>d$, let $\Omega$ be a bounded Lipschitz domain with constants $L$, $r$, $h$. Let $m_{1},m_{2}$ satisfy
\begin{align}
\lambda^{-1} \leq m_{1}(x),m_{2}(x)\leq \lambda \text{ for all } x\in\Omega,\labeleq{ales-ml}\\
\| m_{1}\|_{W^{1,p}(\Omega)}, \|m_{2}\|_{W^{1,p}(\Omega)} \leq E \labeleq{ales-mE}.
\end{align}
Let $A$ be sufficiently bounded and monotone; then, 
\begin{align*}
\|A(m_{1},\cdot) - A(m_{2},\cdot)\|_{L^{\infty}(\partial\Omega)}\leq C \|\Lambda_{A(m_{1},\cdot)} - \Lambda_{A(m_{2},\cdot)}\|_{*}.
\end{align*}
Here $C$ is a constant that depends only on $d, p, L, r, diam(\Omega), \lambda$ and $E$.
\end{theorem}

\begin{theorem}[\cite{Alessandrini2001}, Theorem 2.4]\label{thm:Aless1}
Suppose  $m_{1}, m_{2}$ satisfy \refeq{ales-ml} and \refeq{ales-mE}. Suppose also that $\Omega$ can be partitioned into a finite number of domains $\{\Omega_{j}\}_{j\leq N}$, with $m_{1}-m_{2}$ analytic on each $\overline{\Omega}_{j}$. Then, $\Lambda_{A(m_{2},\cdot)}=\Lambda_{A(m_{2},\cdot)}$ implies that $A(m_{1},\cdot)=A(m_{2},\cdot) \text{ in } \Omega$.
\end{theorem}

Our main result is a direct application of this theory to the inverse homogenization problem of determining $\aeps$ from measurements of homogenized solutions.

\begin{theorem}{} \label{thm:main}Let $a(\cdot, \cdot)$ be a $d\times d$, bounded, symmetric matrix function that is locally periodic, uniformly positive definite and Lipschitz in the first variable. Furthermore, suppose $a(\cdot, \cdot)$ admits a homogenized coefficient $\aeff\in \mc{H}$.  For functions $m_{1}$ and $m_{2}$ satisfying the assumptions of Theorem \ref{thm:Aless1}, define $\aeps_i=a(m_i(x), x/\epsilon)$, $\aeff_i=\aeff(m_i(x), x)$ for $\epsilon>0$ and $i=1, 2$. 

Then, $\Lambda_{\aeff_{1}}=\Lambda_{\aeff_{2}}$ implies that $a^{\epsilon}_{1}=\aeps_{2} \text{ in } \Omega$. Furthermore, there is a constant $C>0$ with
\begin{align}
\|\aeps_{1} - \aeps_{2}\|_{L^{\infty}(\partial\Omega)}\leq C\|\Lambda_{\aeff_{1}}-\Lambda_{\aeff_{2}}\|_{*}\labeleq{mainresult}.
\end{align}
\end{theorem}

\begin{proof}
 A part of the proof of Theorem \ref{thm:Aless1} in \cite{Alessandrini2001} involves showing that for $A\in \mc{H}$ there exists a positive constant $C_{1}$ with  $\|m_{1}-m_{2}\|\leq C_{1}\|\Lambda_{A(m_{1},\cdot)} - \Lambda_{A(m_{2},\cdot)}\|_{*}$. The Lipschitz continuity of $a$ gives the stability result,
 \[\|a(m_{1}, \cdot) - a(m_{2}, \cdot)\|_{L^{\infty}(\partial\Omega)} \leq  C \|m_{1}-m_{2}\|\leq C_{1}\|\Lambda_{A(m_{1},\cdot)} - \Lambda_{A(m_{2},\cdot)}\|_{*}. \]

\end{proof}

Calder\'{o}n's inverse problem is severely ill-posed, even in the case of isotropic coefficients. In order to resolve stability issues (described in \cite{Alessandrini2007a}), some approaches replace a-priori regularity assumptions for $A$ with different assumptions that are better suited for applications. For example, it is known that  if $A$ is a piecewise constant scalar function, the problem is Lipschitz stable. However, even in this case, the stability constant grows exponentially with the number of unknowns \cite{Alessandrini2005}. As a result, the techniques in this paper are applied to parameters $m$ of low dimension.

\subsection{Sufficient conditions for microscale recovery}

The functions $\aeps$ defined in Theorem  \refeq{mainresult} admit a homogenized matrix $A$ that is a symmetric, uniformly positive definite matrix function with bounded elements. The crucial step is to show that $A(m, \cdot)$ is monotone in the sense of \refeq{monotonicity}. 

The homogenized coefficients corresponding to microstructures of type \ref{AMP} and \ref{VF} satisfy the monotonicity condition if there is a constant $E>0$ with 
\begin{align*}
D_{m} \<{a}(m, \cdot) ^{-1} \>^{-1}>E^{-1}\text{ and } D_{m} \<{a}(m, \cdot) \>>E^{-1}.
\end{align*}

The homogenization of microstructures of type \ref{ANG} does not satisfy the monotonicity condition; the matrix \[D_{m}A = \(\langle \hat{a}^{-1}\rangle^{-1}-\langle \hat{a}\rangle\) \begin{pmatrix}-\sin({2\pi m}) &\cos({2 \pi m}) \\
\cos({2 \pi m}) &\sin({2\pi m})\end{pmatrix}\] has eigenvalues $\pm1$. Therefore, Theorem \ref{thm:Aless1} cannot be directly applied. In \S\ref{sec:numericalresults}, we present numerical results for this case. 

The homogenization of cell structures of type \ref{AMP2d} satisfy the monotonicity condition if there is a constant $E>0$ with 
\begin{align*}
D_{m} \<{a}_{1}(m, \cdot) ^{-1} \>^{-1}>E^{-1}\text{ and } D_{m} \<{a}_{2}(m, \cdot) ^{-1} \>^{-1}>E^{-1}.
\end{align*} The homogenized cell structure \ref{VF2D} satisfies the monotonicity requirement  if

\begin{align*}
D_{m}\bar{a}(m)  =  \frac{k_{1}^{2}k_{2}}{(m (k_{2}-k_{1})+k_{1})^{2}}  - k_{2} >E^{-1}
\end{align*}
for all $m\in I_{\lambda}$. Since $0<\lambda^{-1}<m<\lambda<1$, it follows that monotonicity is guaranteed if $k_{1}$ and $k_{2}$ satisfy
\begin{align*}
k_{1}>\sqrt{k_{2}(k_{2}+E^{-1})}.
\end{align*}



It should be noted that the conditions given here are sufficient, but not necessary. In fact, an analogue of Theorem \ref{thm:Aless1} holds in cases when the functions $D_{m}A$ are not strictly monotone \cite{Alessandrini2001}. 

\subsection{Mismatch in boundary measurements}\label{sec:Jm}

The theory given so far justifies the uniqueness and boundary stability of solutions to an inverse homogenization problem of determining a microscale parameter $m$ from macroscopic data. In the numerical experiments, (\ref{microinv}) is solved by matching macroscopic predictions with highly oscillatory data. The justification here is based on the theory of homogenization.

Let $G:L^{\infty}(\Omega)\rightarrow H_{0}^{1}(\Omega)$ be the solution operator corresponding to $\refeq{eit}$.

\begin{theorem}\label{thm:minimizer} 
Let $f^{\epsilon}:I_{\lambda}\rightarrow C^{\infty}(\partial\Omega)$ and  $F:I_{\lambda}\rightarrow C^{\infty}(\partial\Omega)$ be given by
\begin{align*}
f^{\epsilon}(m) = \aeps({m},\cdot) \nabla \ueps({m},\cdot) \cdot \hat{n}, \qquad F(m) = \aeff(m,\cdot)\nabla \ueff(m,\cdot) \cdot \hat{n},
\end{align*} where $\hat{n}$ is a vector that is normal to $\partial \Omega$, $\aeps(m, \cdot)$ is of the form \refeq{aeps-m}, $A(m,\cdot)$ is the homogenized coefficient corresponding to $\aeps(m,\cdot)$,  $\ueps(m,\cdot)=G(\aeps(m,\cdot))$, and $U(m,\cdot)=G(A(m,\cdot))$. 
For a fixed parameter $\bar{m}\in I_{\lambda}$, define the minimization functionals
\begin{align}
\mathcal{J}^{\epsilon}(m) &= \int_{\partial \Omega} \( f^{\epsilon}(\bar{m})-F(m)\)^{2} \varphi ds, \qquad 0<\epsilon<1,\qquad  \varphi\in C^{\infty}(\partial \Omega).\labeleq{Jm}\\
\mathcal{J}(m) &= \lim_{\epsilon\rightarrow 0}\mc{J}^{\epsilon}(m).
\end{align}
Then, $\overline{m}$ is the unique minimizer of $\mathcal{J}$. \end{theorem}

\begin{proof} By expanding the integrand in \refeq{Jm},
\begin{align*}
\mathcal{J}^{\epsilon}(m) &= \int_{\partial \Omega} \(f^{\epsilon }(\bar{m})- F(\bar{m})+ F(\bar{m}) -F(m)\)^{2} \varphi ds\\
&= I_{1}+I_{2}+I_{3},\end{align*}
where 
\begin{align*}
I_{1}&=\int_{\partial \Omega} \(f^{\epsilon }(\bar{m})- F(\bar{m})\)^{2}\varphi ds\\
I_{2}&=2 \int_{\partial \Omega} \(f^{\epsilon }(\bar{m})- F(\bar{m})\)(F(\bar{m}) -F(m)) \varphi ds\\
I_{3}&=\int_{\partial \Omega} \(F(\bar{m}) -F(m)\)^{2}\varphi ds.
\end{align*}

The term $I_{1}$ is independent of $m$.  For the second term, note that since $\ueps$ and $U$ are smooth solutions to \refeq{eit}, for functions $\psi\in C^{\infty}(\Omega)$,
\begin{align*}
\int_{\Omega}\(\nabla\cdot\aeps\nabla\ueps- \nabla\cdot A\nabla U\)\cdot \psi ds = 0,\end{align*}
and therefore, by Green's theorem,
\begin{align*} \int_{\Omega}\(\aeps\nabla\ueps- A\nabla U\)\cdot \nabla \psi ds = \int_{\partial\Omega}\(\aeps\nabla\ueps\cdot{\hat{n}}- A\nabla U\cdot \hat{n}\) \psi ds.
\end{align*}
Homogenization theory gives the convergence $\aeps\nabla\ueps \rightharpoonup \aeff({\overline{m}})\nabla\ueff({\overline{m}})$ weakly in $L_{2}(\Omega)$ as $\epsilon\rightarrow 0$. Since $\psi =(F(\bar{m}) -F(m))\varphi$ is smooth, it follows that
\begin{align*}
 \lim_{\epsilon\rightarrow 0}\int_{\partial \Omega}\(f^{\epsilon }- F(\bar{m})\)\(F(\bar{m}) -F(m)) \)\varphi ds 
= 0.
\end{align*}
Therefore, $\mathcal{J}(m)=\lim_{\epsilon\rightarrow 0}\mathcal{J}^{\epsilon}(m)=I_{3}$. The term $I_{3}$ is minimized  when $F(\bar{m})=F(m)$, and it follows from the results in the previous section the minimizer is $m=\overline{m}$.
\end{proof}

\section{HMM for the forward problem}\label{sec:hmm}

In the simulations of the macroscopic forward model \refeq{eff}, numerical homogenization is performed using the finite element heterogeneous multiscale method (FE-HMM). Here we provide a brief presentation of the scheme; further details of various HMM formulations can be found in \cite{Abdulle2012, E2003, Ming2004a}.

The FE-HMM scheme is designed for approximations of the homogenized equation \refeq{eff} when the coefficients in the effective model are not known explicitly. By employing a microscale solver on local subdomains,  the homogenized coefficients can be estimated in an efficient way.

The macroscopic solver is the traditional $\mathcal{P}_{k}$ finite element method on a coarse triangulation $\mc{T}_{H}$ of the domain containing elements of size $H>\epsilon$. The macroscale bilinear form is defined for functions $V$ and $W$ lying in the finite element space $X_{H}$,  
\begin{align}\mathcal{B}(V,W) := \int_{\Omega}  \nabla V \cdot \ahmm(x) \nabla W dx+ \int_{\Omega}b W V dx, \labeleq{hmmbl}\end{align}
 where $\ahmm$ is not known explicitly. The first integral in \refeq{hmmbl} is approximated  using numerical quadrature points $\{x_{l}\}$ and weights $\{\omega_{l}\}$,
\begin{align}
\int_{\Omega}  \nabla V \cdot \ahmm(x) \nabla W dx\simeq \sum_{K\in \mc{T}_{H}}|K|\sum_{x_{l}\in K}\omega_{l}\(\nabla V \cdot \ahmm \nabla W\)(x_{l}),\labeleq{hmmquad}
\end{align}
where $|K|$ is the measure of $K$.

The stiffness matrix entries are estimated at each quadrature point $x_{l}$ by using a microscale solver on subdomains $I_{\delta}(x_{l}):=x_{l}\pm \tfrac{\delta}{2} I$.  Then effective behavior of $\aeps$ is captured locally through the solution of cell problems
 \begin{align}
 -\nabla\cdot( a^{\epsilon} \nabla v_{l}^{\epsilon})  = 0 \text{ in } I_{\delta}(x_{l}),\qquad v_{l}^{\epsilon}=V_{l}\text{ on }\partial I_{\delta}(x_{l}), \label{eq:cell}
\end{align}
where $V_{l}$ is the linear approximation of $V$ at $x_{l}$. 

Again, a standard $\mathcal{P}_{k}$  finite element solver is used on a triangulation $\mc{T}_{h}^{l}$ of the subdomains. The spacing $h<\epsilon$ is chosen sufficiently small in order to resolve the microscale. 
Figure \ref{fig:HMMmesh} contains a diagram of the macro-micro grid coupling in a typical FEM-HMM formulation.

Then, the term $(\nabla V \cdot \ahmm \nabla W)(x_{l})$ in \refeq{hmmquad} can be estimated by
\begin{align*}(\nabla V \cdot \ahmm \nabla W)(x_{l}) \simeq \frac{1}{\delta^{d}} \int_{I_{\delta}(x_{l})}\nabla v^{\epsilon}_{l}\cdot (a^{\epsilon} \nabla w^{\epsilon}_{l}) dx.  
\end{align*} 
\noindent The HMM bilinear form is then defined by
\begin{align*}
\mathcal{B}_{HMM}(V,W):= \sum_{K\in \mc{T}_{H}}{|K|}\sum_{x_{l}\in K}\omega_{l}\(\ \frac{1}{\delta^{d}}\int_{I_{\delta}(x_{l})}\nabla v^{\epsilon}_{l}\cdot (a^{\epsilon} \nabla w^{\epsilon}_{l}) dx +\(b W V\)(x_{l})\).
\end{align*}
\noindent Finally, we have that the HMM solution, $\uhmm \in g+X_{H}$, solves $\mc{B}_{HMM}(V, V) = (f, V)$, for all $V\in X_{H}$.

\begin{figure}[htbp]
\caption{{\bf FEM-HMM Discretization.} An illustration of macro-micro coupled grids used in FE-HMM for elliptic PDEs.}
\begin{center}
\tikzstyle{mybox} = [draw=black!40 fill=lgray, very thick,
    rectangle, rounded corners, inner sep=10pt, inner ysep=20pt]
\tikzstyle{fancytitle} =[fill=black!40, text=white]

\begin{tikzpicture}[scale=1,every node/.style={minimum size=1cm},on grid]
		
    \begin{scope}[ xshift=-163,every node/.append style={
    yslant=0.5,xslant=-1},yslant=0,xslant=.0
            ]

 \pgftransformcm{0}{1}{\xcoord}{\ycoord}{\pgfpointorigin} 
    \path[clip, preaction={draw=black, very thick}] (5,0) -- (0,0) -- (0,5) -- (5,5) -- cycle;
    \foreach \x in {0,1,2,...,\rows} {
     \draw[black,thick, dashed] (0,\x) -- (\x,0);
        \draw[black,thick, dashed] (\x,0) -- (\x,\rows);
        \draw[black,thick, dashed] (0,\x) -- (\rows,\x);

        \foreach \y in {0, 1, 2, ..., \rows}{
        
         \draw[step=.05mm, black!50,thin] (\x+.15,\y+.15) grid (\x+.35,\y+.35);  
        \filldraw (\x+.25,\y+.25) circle (.5pt);
         \draw[step=.05mm, black!50,thin] (\x+.65,\y+.65) grid (\x+.85,\y+.85);  
        \filldraw (\x+.75,\y+.75) circle (.5pt);
           \draw[black!40, thick] (\x+.15,\y+.15) rectangle (\x+.35,\y+.35);  
         \draw[black!40, thick] (\x+.65,\y+.65) rectangle (\x+.85,\y+.85);  
         }
    }

    \end{scope}
    \begin{scope}[
    	xshift=60,every node/.append style={
    	    yslant=0,xslant=0},yslant=0,xslant=0 ]
        \fill[white,fill opacity=.9] (0,0) rectangle (3,3);
       
 \pgftransformcm{0}{1}{\xcoord}{\ycoord}{\pgfpointorigin} 
   \path[clip, preaction={draw=black!40, very thick}] (3,0) -- (0,0) -- (0,3) -- (3,3) -- cycle;
    \foreach \x in {0,1,2,...,\rowsmicro} {
     \draw[black!40] (0,\x/3) -- (\x/3,0);
        \draw[black!40] (\x/3,0) -- (\x/3,\rows/3);
        \draw[black!40] (0,\x/3) -- (\rows/3,\x/3);
    } 
       \draw[black!40,very thick] (0,0) rectangle (3,3);
          \filldraw (1.5,1.5) circle (1pt) node[align=left,below] {$x_{l}$};

  \end{scope}

    \draw[black!40, dashed] (-1.1,.65) -> (2.1,0);  
   \draw[black!40, dashed] (-1.1,.85) -> (2.1,3);  
   \draw[black!40, dashed] (-.9,.65) -> (5.15,.05);  
   \draw[black!40, dashed] (-.9,.85) -> (5.1,2.9);  
%

\end{tikzpicture}%
\label{fig:HMMmesh}
\end{center}
\end{figure}
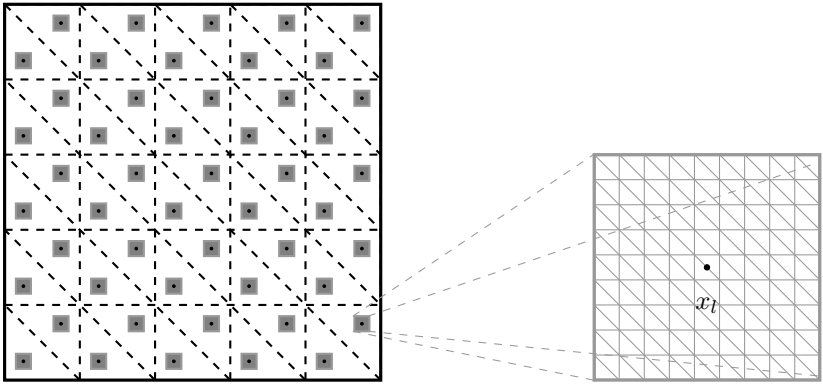

\subsection{Errors in forward modeling}

An analysis of the errors involved in the FE-HMM formulation for elliptic problems is found in \cite{Ming2004a}. The main result is the following theorem, assuming a $k$th order numerical quadrature scheme for $\refeq{hmmquad}$ that satisfies
\begin{align*}
\frac{1}{|K|}\int_{K}p(x)dx=\sum_{l=1}^{L}\omega_{l}p(x_{l})\quad \text{for all } p(x)\in \mathcal{P}_{2k-2},
\end{align*}
where $\omega_{l}>0$, $l=0, \hdots, L$.
\begin{theorem}[\cite{Ming2004a}, Theorem 1.1] Denote by $U\in H^{1}_{0}(\Omega)$, $\uhmm\in X_{H}$ the solutions to \refeq{eff} and the FE-HMM solution, respectively. Let
\begin{align*}
e(HMM) = \underset{x_{l}\in K, K\in \mc{T}_{H}}{\max}\|\aeff(x_{l}) - \ahmm(x_{l})\|,
\end{align*} 
where $\|\cdot\|$ is the Euclidean norm. If $\ueff$ is sufficiently smooth, and $\lambda I \leq \aeps \leq \Lambda I$ for $\lambda, \Lambda>0$, then there exists a constant $C$ independent of $\epsilon, \delta$ and $H$ such that
\begin{align*}
\|\ueff - \uhmm\|_{1}\leq C \(H^{k}+e(HMM)\),\\
\|\ueff - \uhmm\|_{0}\leq C \(H^{k+1}+e(HMM)\).\\
\end{align*}

\end{theorem}

Then $\uhmm\rightarrow U$ as $e(HMM)\rightarrow 0$. For the periodic homogenization problem it is also shown that 
\begin{align*}
e(HMM)\leq \begin{cases}
C\epsilon & I_{\delta}(x_{l})= x_{l}+\epsilon I\\
C(\frac{\epsilon}{\delta}+\delta) & \text{otherwise}.
\end{cases}
\end{align*}

A comparison of errors from using a purely macroscale solver and HMM is given in Figure \ref{fig:erreps} and Figure \ref{fig:errcellsize}. The full solution $\ueps$ to \refeq{ms} with $b=0$ and $f=1$ subject to Dirichlet boundary conditions $\ueps|_{\partial\Omega}=0$, is computed using direct numerical simulation on a fine mesh with element size $h=1/800$. Solutions $U$ of the homogenized equation \refeq{eff} are resolved on a coarse resolution mesh of element size $H=1/20$. We denote by  $U_{HOM}$ and $U_{HMM}$ the approximations of $U$ using analytic formulas and HMM, respectively.

 \newcommand{\legend}{\begin{center}
 \begin{tikzpicture}
\begin{customlegend}[legend columns=-1,legend style={draw,column sep=2ex}, legend entries={$\|\ueps - U_{HOM}\|_{L_2}$,$\|\ueps - U_{HMM}\|_{L_2}$ }]
    \addlegendimage{dashed, black,fill=black,sharp plot}\addlegendimage{black,fill=black,sharp plot}
    \end{customlegend}
\end{tikzpicture}\end{center}}

\begin{figure}
\label{fig:erreps}
\caption{Errors in the macroscopic solution  $\|\ueps - U\|_{L_{2}}$ as $\epsilon\to0$ using a purely macroscale solver (dashed) and FE-HMM (solid). HMM microscopic cells $I_{\delta}$ are of  size $\delta = 10\epsilon$.}
\begin{center}

\plotepserr{\amplitudeeps}{A - Amplitude}\plotepserr{\areafractioneps}{B - Volume Fraction}

\end{center}

\end{figure}

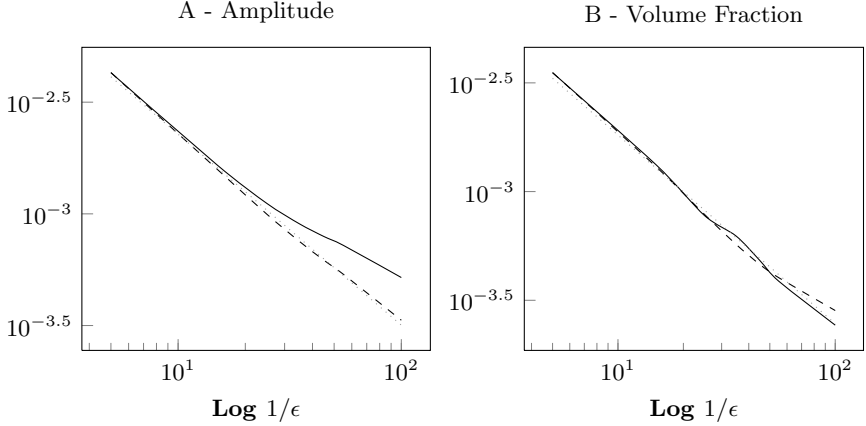
\begin{figure}
\begin{center}
\caption{Errors in macroscopic solution using a purely macroscale solver (dashed) and FE-HMM (solid) with  by varying HMM microscopic cell size $\delta$. Here $\epsilon=1/100$.}
    \newcommand{\plotcellsize}[2]{
\begin{tikzpicture}
\begin{axis}[width=.5\linewidth,
        xlabel={HMM cell size: ${\delta }/{\epsilon}$},
        plotstyle,
        ymin=0,
        xmin=2, xmax=10,
        xtick={2,4, ..., 10},
         title style={yshift=-1ex},
         title={#2}]
         \addplot[smooth,dashed, mark=.,black] table[x=cs,y=eff]  {#1};
\addplot[smooth,mark=.,black] table[x=cs,y=hmm] {#1}  ;
   \end{axis}
    \end{tikzpicture}}

  \plotcellsize{\amplitudecellsize}{A - Amplitude}\plotcellsize{\areafractioncellsize}{B - Volume Fraction}
\label{fig:errcellsize}
\end{center}
\end{figure}

In the case of unknown or random microstructure, HMM can be performed ``on the fly'', and computational time can be reduced using parallel solvers for the local cell problems. For periodic problems in two dimensions (as in this work), precomputing the cell problem solutions increases the efficiency of HMM.


\section{Numerical experiments}\label{sec:numericalresults}

In this section we present results of numerical simulations that demonstrate parameter inversion of elliptic equations  using homogenization theory and the ideas discussed in \S\ref{sec:msinv}. It is assumed that the parameter $\epsilon$, as well as the mapping $m\rightarrow \aeps(m)$ is known.

There is no additional regularization of the problem. 
The inverse problem (\ref{microinv}) is solved by minimizing the cost functional \refeq{Jm}
 It should be noted that stability is guaranteed only on the boundary of the domain.

 \begin{figure}
\caption{{\bf Synthetic data and predictions}. The plot contains the graph of the oscillatory Neumann data $\obs({\aeps}, \ueps)=\aeps\nabla\ueps\cdot n\vert_{\Gamma=\{0\leq x \leq 1, y=0\}}$ corresponding to the Dirichlet boundary condition $u|_{\partial \Omega} =\frac{1}{\sqrt{2\pi}}e^{-(x-1)^2}$, where $\aeps$ is given by \refeq{angle}. The macroscopic predictions $\obs({\aeff}, \ueff)=\aeff\nabla\ueff\cdot n\vert_{\Gamma}$ are also plotted. 
}
\label{fig:data}
\begin{center}
{\includegraphics[width=1\textwidth]{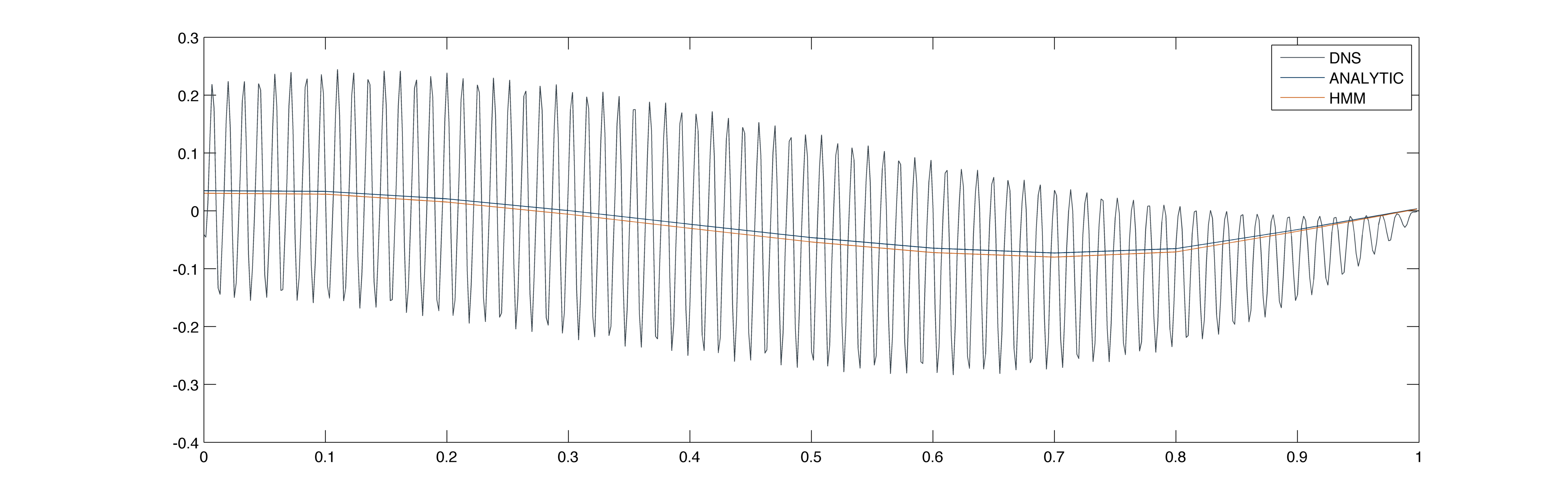}}
   \end{center} 

\end{figure}

In all of the simulations, a standard $\mc{P}^1$ finite element method is used on a regular triangulation of the domain. The \textsc{MATLAB} routine \texttt{lsqnonlin} is used to minimize the least-squares functional \refeq{Jm}.  The synthetic data  is generated using direct numerical simulation  of the full model using a fine mesh with resolution $h<\epsilon$. Macroscopic predictions of the forward model are computed using a coarse mesh with resolution $H>\epsilon$. The local subdomains in the HMM solver are discretized on a fine mesh with spacing $\delta<\epsilon$. This provides a framework for microscale inversion that avoids the major pitfalls of committing an ``inverse crime''. Unless otherwise stated, we set  $\Omega=[0,1]\times[0,1]$, $\epsilon=1/80$, $H=1/10$, $\delta = 3\epsilon$, and $h=1/600$.

The microstructure models \ref{AMP}, \ref{VF}, and \ref{ANG}, respectively, are represented by the multiscale functions
\begin{align}
\aeps_{A}(m(x),x) &= 1.1 +m(x) \sin (2\pi x_{2}/\epsilon), \\
\aeps_{B}(m(x),x) &= .5+2\chi_{\{x_{2}<m\}}(x), \\
\aeps_{C}(m(x),x) &= 1.1 + \sin (2\pi \tilde{x}_{2}/\epsilon), \qquad \tilde{x}=\sigma_{m}x.
\end{align}

\subsection{Inverse conductivity \text{(b = 0)}}

Based on the results in section \ref{sec:Jm}, if $\aeff$ is the homogenized coefficient corresponding to $\aeps$, then $\Lambda_{\aeps}\sim \Lambda_{A}$. We define the measurement operator $\mathscr{G}$ in terms of weak solutions of  $\nabla \cdot (a\nabla u_{k}) = 0$, subject to Dirichlet boundary conditions $u_{k}\vert_{\partial\Omega}=g_{k}$, $1\leq k\leq K$, \[\mathscr{G}(a, u_{k})_{j} = l_{j}(a \nabla u_{k} \cdot \hat{n}), \qquad j=1, \hdots, n.\] Here the linear functionals $l_{j}(f) = f(x_{j})$ are defined for a given set $\{x_{j}\}_{j=1}^{n} \subset \partial \Omega$ and $\{g_{k}\}$ is the set of functions $\{x, y, x^{2}, y^{2}, xy\}$ (see Figure \ref{fig:data}). Coarse meshes can be used to resolve solutions with these boundary conditions.

Table \ref{table:inverr} shows the relative error $|\hat{m} - m|/| m|$ in the estimation of the microscale parameter $m\equiv\theta\in \RR$ and Table \ref{table:invtime} contains a comparison of the performance time using different forward solvers. The differences in the inversion results can be attributed to the resolution of the meshes used, errors introduced by the optimization routine, and the mismatch in scales between the oscillatory data and the slowly varying predictions. 

\begin{table}[h!]
\caption{Relative error in inversion for a microscale parameter $m\equiv\theta\in \RR$.}\label{table:inverr}
\begin{center}
\begin{tabular}{|c|c|c|c|}
  \hline
 &  HMM &  Analytic  & Two-stage \\
   \hline
 Model A &0.04563540  &0.02556686 	 	& 0.02344292 \\
 Model B &0.03623084 	& 0.02234864 	 &0.06537146 \\
 Model C &0.05210436  &0.00726006 	 	 &0.05578068\\
 Model D &0.05093572 &0.00354686 	  	 &0.15121385 \\
 Model E &0.07446168 &0.05607627 	 	 &0.01708140\\
   \hline
\end{tabular}
\end{center}
\end{table}%
\begin{table}[h!]
\caption{Performance time (in seconds) of inversion for a microscale parameter $m\equiv\theta\in \RR$.}\label{table:invtime}
\begin{center}
\begin{tabular}{|c|c|c|c|}%
  \hline
 &  HMM &  Analytic & Two-stage\\
   \hline
 Model A &9.42& 8.86 &  22.13  \\
 Model B &18.30&16.60 &  27.45 \\
 Model C &14.90& 17.38 &  16.47  \\
 Model D &10.27& 10.09 &  17.15  \\
 Model E &15.88& 16.36 &  19.64 \\
   \hline
\end{tabular}
\end{center}
\end{table}%

The analytic and HMM solver perform similarly for all three microstructure models. The longer performance time using the two-stage solver can be attributed to first stage, where inversion for the unknown matrix coefficient $\aeff$ involves three times as many unknowns as direct inversion.

\subsubsection{Representation of the microscale parameter}

\begin{figure}[t]
\includegraphics[width=\linewidth,height=.2\linewidth]{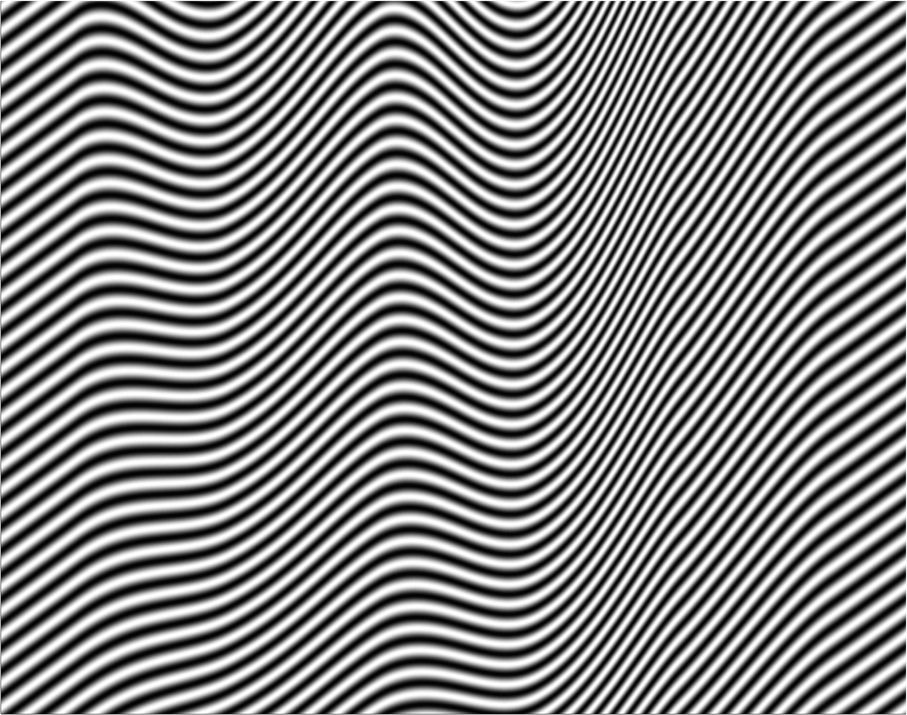}
\caption{  Model \ref{ANG} parametrized using a continuous function $m(x)$ with $N=6$ degrees of freedom.}
  \label{fig:mcontinuous}
\end{figure}
In one experiment, we restrict $m(x)$ to the space of cubic spline interpolants corresponding to the values given by the vector $\theta\in \RR^{N}$. An example of a microstructure with this kind of parametrization is shown in Figure \ref{fig:mcontinuous}.  Forward predictions are made using HMM ``on the fly''. Table \ref{table:mcontinuous} contains the relative errors $\|m-\hat{m}\|/\|m\|$ in the recovered parameter for different values of $N$. 
\begin{figure}
\includegraphics[width=\linewidth,height=.2\linewidth]{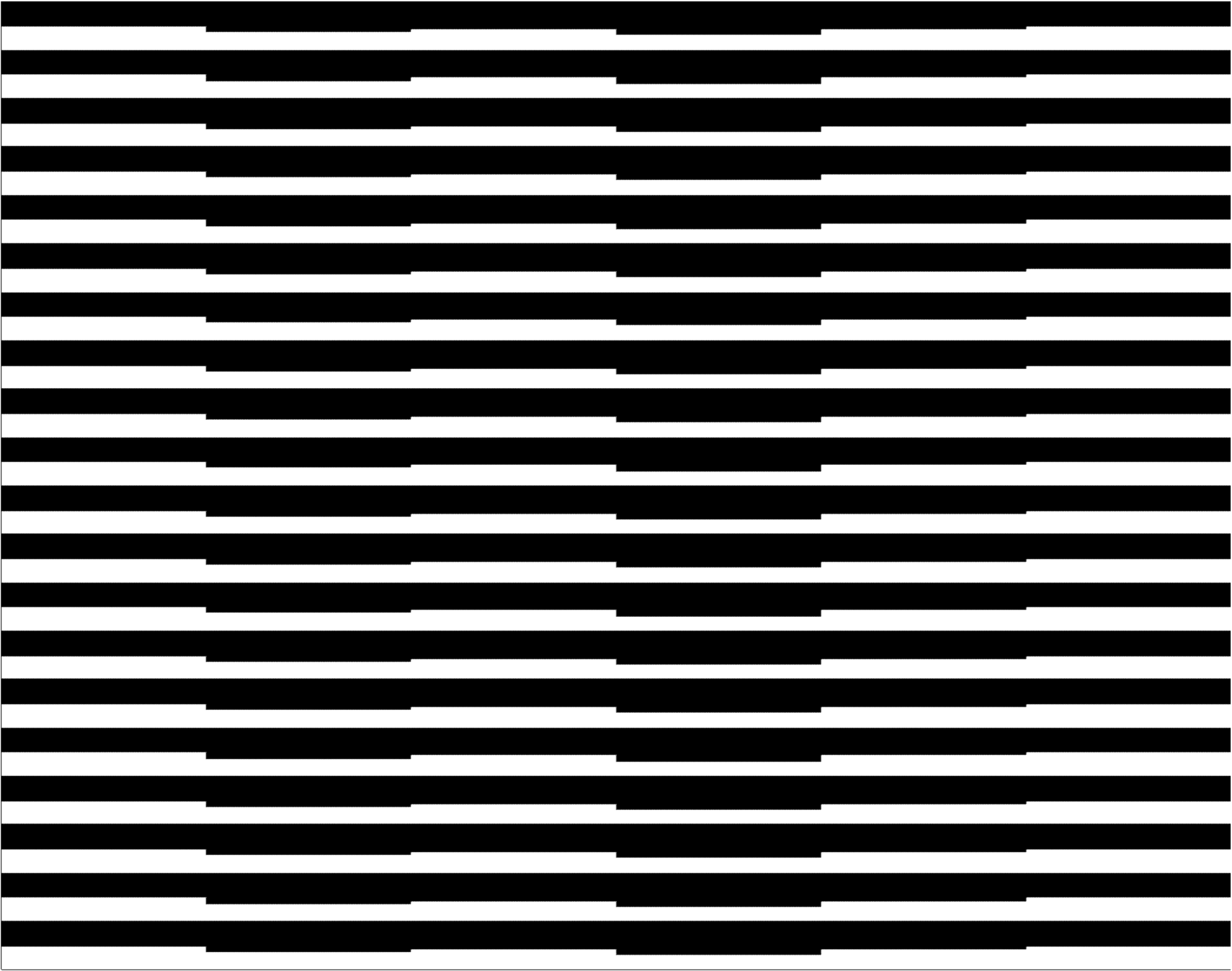}
\caption{  Model \ref{VF} parametrized using a piecewise constant function $m(x)$ with $N=6$.}
  \label{fig:mdiscontinuous}
\end{figure}
\begin{table}[h!]
\caption{Inversion error in $\theta\in \RR^N$ for continuous $m(x)$.}\label{table:mcontinuous}
\begin{center}
\begin{tabular}{|c|c|c|c|}
\hline
$N$ & \ref{AMP} - Amplitude & \ref{VF} - Volume Fraction & \ref{ANG} - Angle\\
\hline
1		&  0.04237400 & 0.04507309  &  0.05701372 	 \\
2		& 0.05485636  & 0.04873258 &	 0.04293228 	\\
3		& 0.05552983  & 0.06892129 &	 0.06720150	 \\
4		& 0.06572240  & 0.05887249 &  0.05945569 	\\
5		& 0.06691606  & 0.07173517 &	 0.08507094  	\\
6		& 0.06761214  & 0.07921053  &	 0.09011505  	\\ 

\hline
\end{tabular}
\end{center}
\end{table}%
\begin{table}[h!]
\caption{Inversion error in $\theta\in \RR^N$ for piecewise constant $m(x)$.}\label{table:mdiscontinuous}
\begin{center}
\begin{tabular}{|c|c|c|c|}
\hline
$N$ & \ref{AMP} - Amplitude & \ref{VF} - Volume Fraction & \ref{ANG} - Angle\\
\hline
1		& 	0.04563540 	&   	0.02234864 &	0.05210436 	  \\
2		&	0.05786556 &	 0.02244457 &  	0.05886389 	  \\
3		&	0.06187033 &	 0.04806429 &	0.06601630 	  \\
4		&	0.07288481	&	0.07027316 &	0.08523131 	 \\ 
5		& 	0.06697655	&	0.09667535 	  &	0.08324536 	\\  
6	 	& 	0.08680081	&	0.07828320 &	0.08282462 	\\  
\hline
\end{tabular}
\end{center}
\end{table}%
Another parameter space is the set of piecewise constant functions $m(x)$ (see Figure \ref{fig:mdiscontinuous}). Forward computations are made using a HMM solver that efficiently makes use of precomputed values of ${A}(m)$. Table \ref{table:mdiscontinuous} contains the errors in the recovered parameter for different values of $N$.

We can extend the ideas in previous sections to unknown parameters $m$ of the form, \[m(x)=(m_{1}(x), \hdots, m_{M}(x)),\] where $M$ is the number of microscale features to be recovered. For the general problem, each function $m_{i}(x)$, $1\leq i \leq M$, is assumed to be a scalar function with $N$ degrees of freedom (see Figure \ref{fig:aepsMN}).

Oscillatory functions describing
models Amplitude-Angle (\ref{AMP}-\ref{ANG}), Volume Fraction-Angle
(\ref{AMP}-\ref{ANG}) and Amplitude-Volume Fraction (\ref{AMP}-\ref{VF}), respectively, are

\begin{align}
\aeps_{AC}(m_{1},m_{2}, x) &= \aeps_{A}(m_{1}, \tilde{x}), \qquad \tilde{x}=\sigma_{m_{2}}x,\\
\aeps_{BC}(m_{1},m_{2}, x) &= .5+2\chi_{\{(0,m_{1})\}}(\tilde{x}_{2}), \qquad \tilde{x}=\sigma_{m_{2}}x,\\
\aeps_{AB}(m_{1},m_{2}, x) &= .5+2m_{1}\chi_{\{x_{2}<m_{2}\}}(x).\end{align}

 \begin{figure}
\begin{center}
 \includegraphics[width=.3\linewidth]{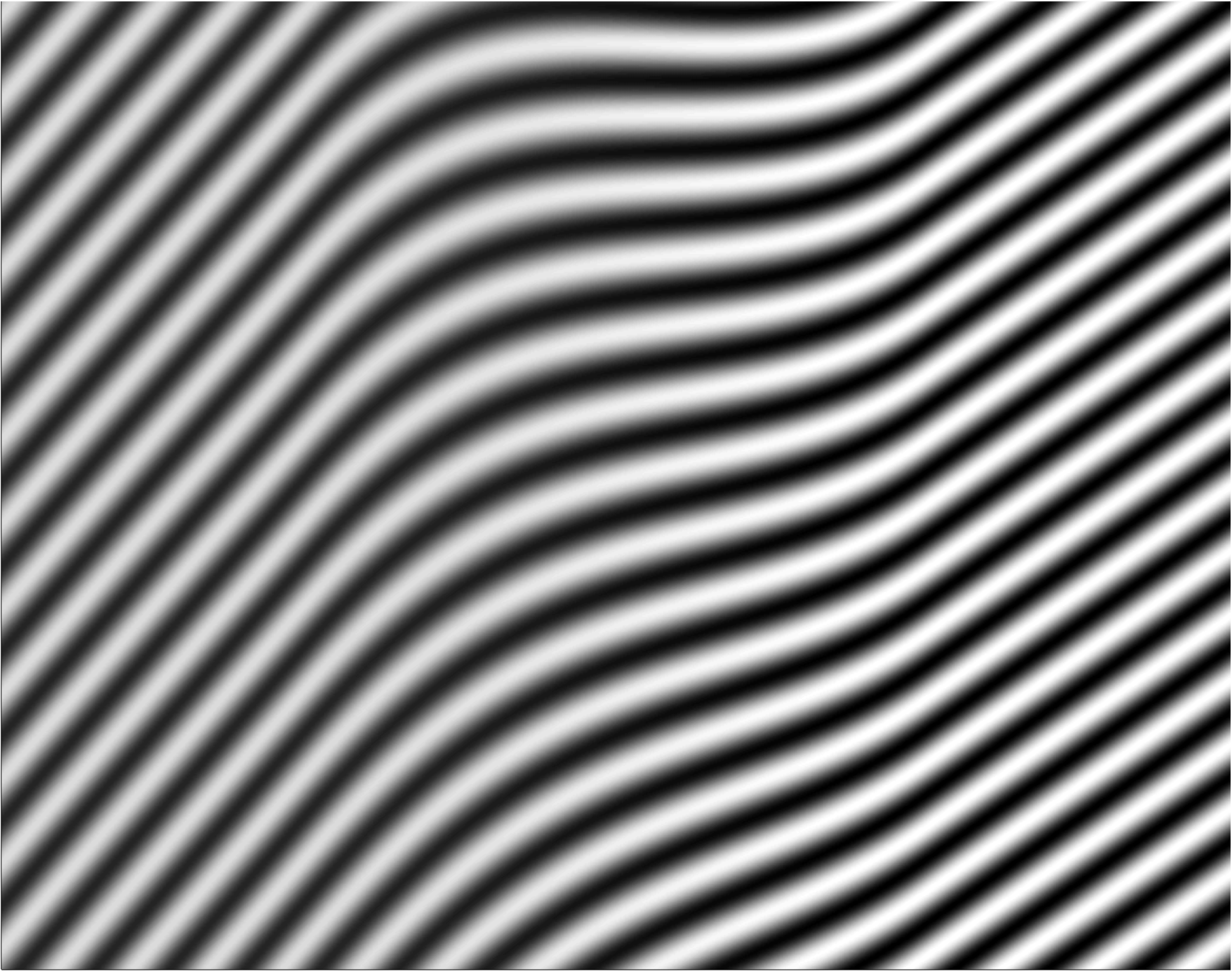}\hfill \includegraphics[width=.3\linewidth]{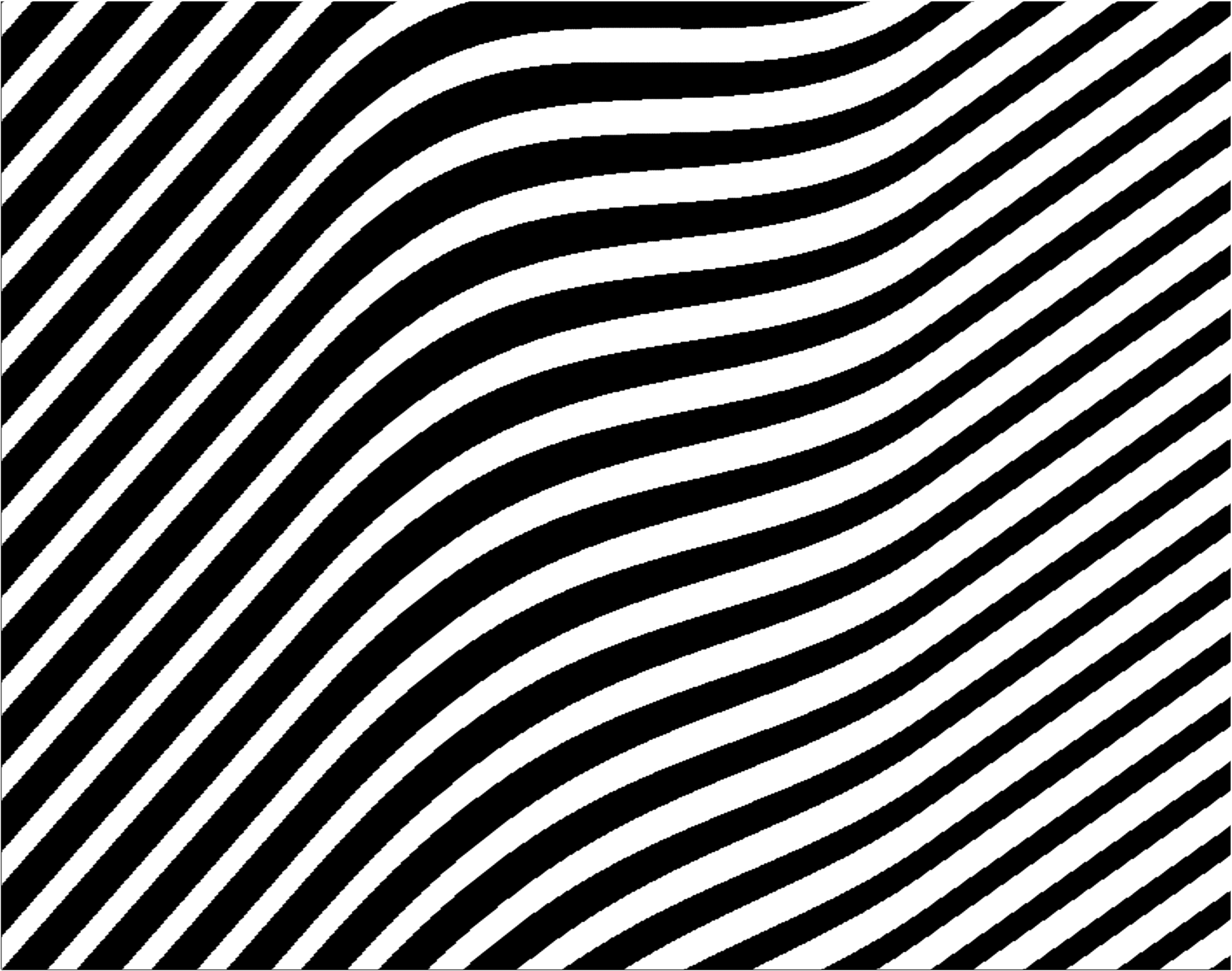} \hfill\includegraphics[width=.3\linewidth]{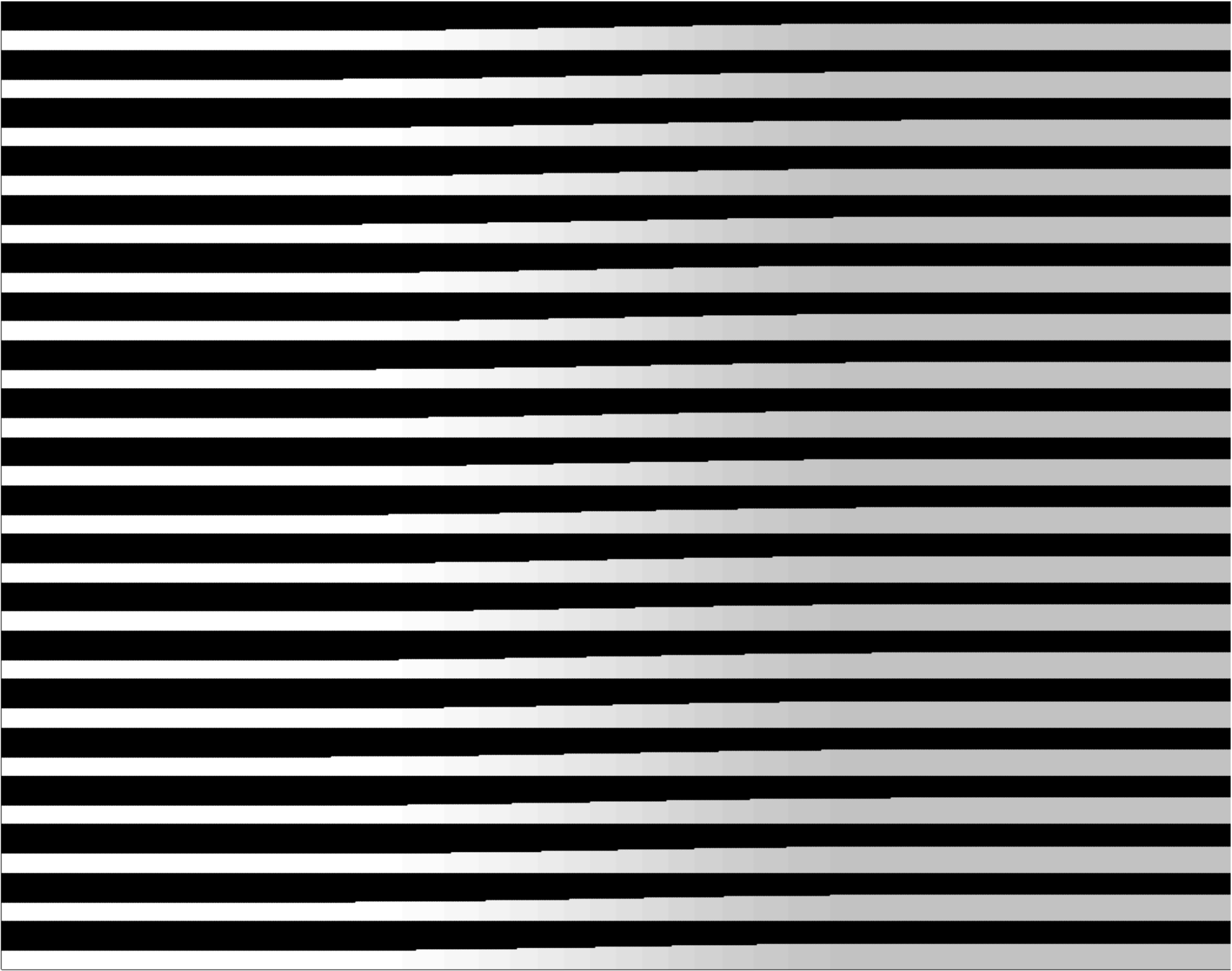} 
\caption{Microstructure models for continuous vector fields $m(x) = (m_{1}(x), m_{2}(x))$. Left to right: Amplitude-Angle,  Volume Fraction-Angle, Amplitude-Volume Fraction.}\label{fig:aepsMN}
\end{center}
\end{figure}

Therefore, the inverse problem reduces to determining a finite dimensional vector of unknowns, $\theta\in \RR^{MN}$. Table \ref{table:mcontinuous2M} contains the errors from the numerical experiments for $M=2$.

\begin{table}[h]
\caption{Inversion error in $\theta\in \RR^{2N}$ for continuous $m(x) = (m_{1}(x), m_{2}(x))$.}\label{table:mcontinuous2M}
\begin{center}
\begin{tabular}{|c|c|c|c|}
\hline
$N$ & Amplitude-Angle & Volume Fraction-Angle& Amplitude-Volume Fraction\\
\hline
1		&  0.03082785 & 0.04804962   &   0.04394855 	 \\
2		& 0.12689231  & 0.07667133 &	 0.06829348  	\\
3		& 0.11401958  & 0.09285163 &	 0.07980421 \\
\hline
\end{tabular}
\end{center}
\label{table:aepsMN}
\end{table}%

In certain cases, the solutions to cell problems corresponding to different multiscale coefficients $\aeps$ are indistinguishable. In particular, for a fixed $m$, there exists a $\tilde{m}$ such that the homogenized coefficients corresponding to microstructures with a parametrized volume fraction $m\rightarrow\aeps_{VF}(m)$ are equal to the homogenized coefficients corresponding to microstructures with a parametrized amplitude $\tilde{m}\rightarrow\aeps_{AMP}(\tilde{m})$. However, if the recovered parameter is constrained to a convex subset of the search space, the problem can be reformulated to guarantee a unique recovery.


\subsubsection{Random microstructure}
We consider a model of layered media where the microstructure is represented by a random function $m(x)\rightarrow \aeps(m(x),x,\omega)$, where
\begin{align}
 \aeps(m,x, \omega)&=a(m,  X_\epsilon(x, \omega)), \labeleq{aepsomega}\\
 X_\epsilon(x, \omega)&=\sum_{j=0}^{\lfloor1/\epsilon\rfloor}X_j(\omega) \chi_{\left[j\epsilon, (j+1)\epsilon\right)}(\tilde{x}), \qquad \tilde{x}=(\sigma_{\pi/4}x)_{2},
\end{align}
where  $\omega$ is an element of a sample space $\mathcal{X}$ and $X_j$ are independent, random variables that are uniformly distributed on the interval $[-1, 1]$. Figure \ref{fig:aepslayersrand} shows a plot of \refeq{aepsomega} for  $a(m, \xi) = 1 +m\xi$.

In these experiments, we fix $m\equiv \theta\in (0,1)$ and minimization of the least-squares functional \refeq{Jm} is performed for 100 realizations of $\aeps(\theta, \omega)$. In each trial, the same realization is used to generate both the data, $\data = \obs(\aeps, \ueps_{k})_{j} = l_{j}(\aeps\nabla \ueps_{k} \cdot \hat{n})$ and the predictions $z = \obs(\aeff, U_{k})$. 

We compare the performance of HMM forward solvers corresponding to three different choices of the size of the local subdomains $I_{\delta}$; $\delta=2\epsilon$, $\delta=4\epsilon$, and $\delta=8\epsilon$. Table \ref{table:inverrrand} contains the frequency of recovered parameters $\hat{\theta}$ that lie in the interval $E_{\theta}$ centered at the true parameter. The results are consistent with the expectation that the accuracy of the parameter estimation using HMM would improve with increased cell size.

\begin{figure}
\begin{center}
\includegraphics[height=.3\linewidth, width=\linewidth]{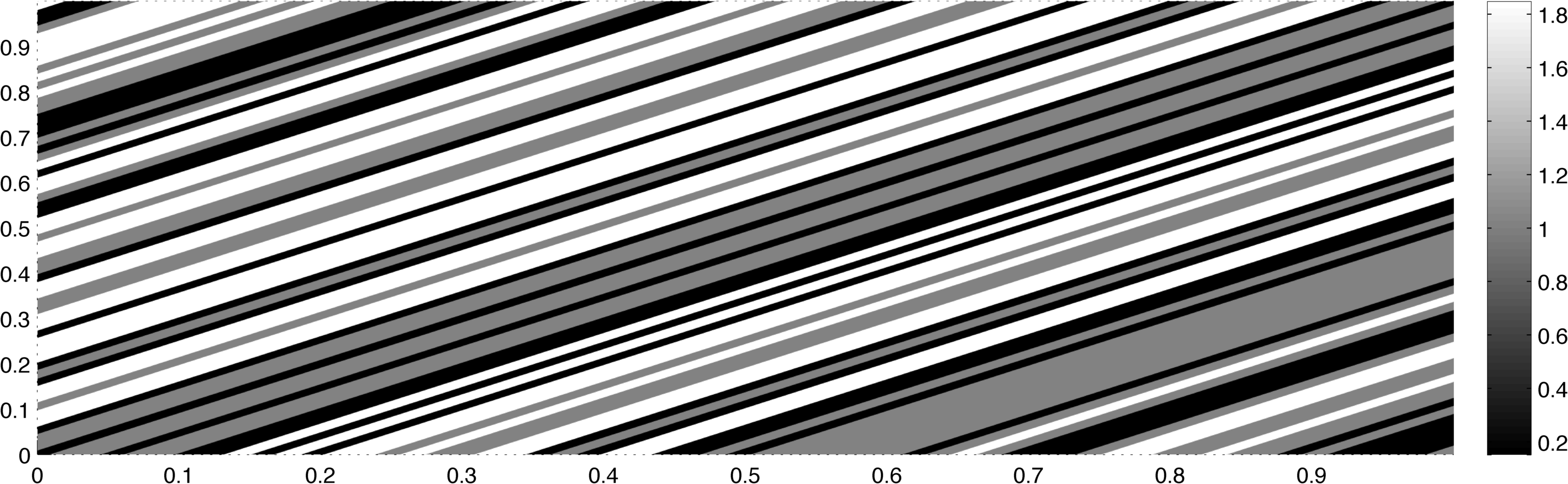}
\caption{{Random microstructure in layered materials.}}  \label{fig:aepslayersrand}
\end{center}

\end{figure}

\begin{table}[h]
\caption{Microscale parameter inversion for a random microstructure. The true parameter is $\theta=.8$, and the frequency of recovered parameters $\hat{\theta}$ lying in the interval $E_{\theta}$ is given.}
\begin{center}
\begin{tabular}{|c|c|c|c|}
   \hline
 $E_{\theta}$ &  $\delta = 2\epsilon$ &  $\delta = 4\epsilon$& $\delta = 8\epsilon$\\
   \hline
 $(.7, .9)$& $16\%$ & $26\%$ & $57\%$\\
 $(.75, .85)$& $4\%$ & $14\%$ & $29\%$\\
  $(.79, .81)$& $0\%$ & $2\%$ & $11\%$\\
     \hline
\end{tabular}
\end{center}
\label{table:inverrrand}
\end{table}%

\subsubsection{Noisy data}

Here, measurement error is introduced in the observations, 
$\data = \obs(\aeps, \ueps_{k})_{j} = l_{j}(\aeps\nabla \ueps_{k} \cdot \hat{n}) (1 +\xi)$,
where $\xi$ is a normally distributed random variable with mean zero and standard deviation $\sigma =.1$.  Figure \ref{fig:noisydata} contains histograms of the relative errors in the recovered parameter $m\equiv\theta \in \RR$. From the experiments it is clear that the two-stage procedure  resulted in errors with a larger variance than the errors from direct inversion. Modifications of this procedure will be needed in order to improve robustness to noisy input data.

\begin{figure}
\caption{{\bf Microscale inversion with noisy observations (N=M=1)}. The histogram shows the results of 100 trials of microscale inversion with measurement error of $10\%$ added to the synthetic data.}\label{fig:noisydata}
\begin{center}
\invplothist{\amplitudeinvnoise}{A - Amplitude}\invplothist{\areafractioninvnoise}{B - Volume Fraction}\invplothist{\angleinvnoise}{C - Angle}\parbox[b][.5\linewidth][c]{.44\linewidth}{
\begin{center}
\begin{tikzpicture}
\begin{customlegend}[legend columns=1,legend style={draw,column sep=1ex}, legend entries={Analytic Solver,Two-Stage Solver,HMM}]
    \addlegendimage{area legend,fill=black!20}
    \addlegendimage{area legend,fill=black!80}
    \addlegendimage{area legend, fill=white}
    \end{customlegend}
\end{tikzpicture} \end{center}}
  \end{center}
\end{figure}
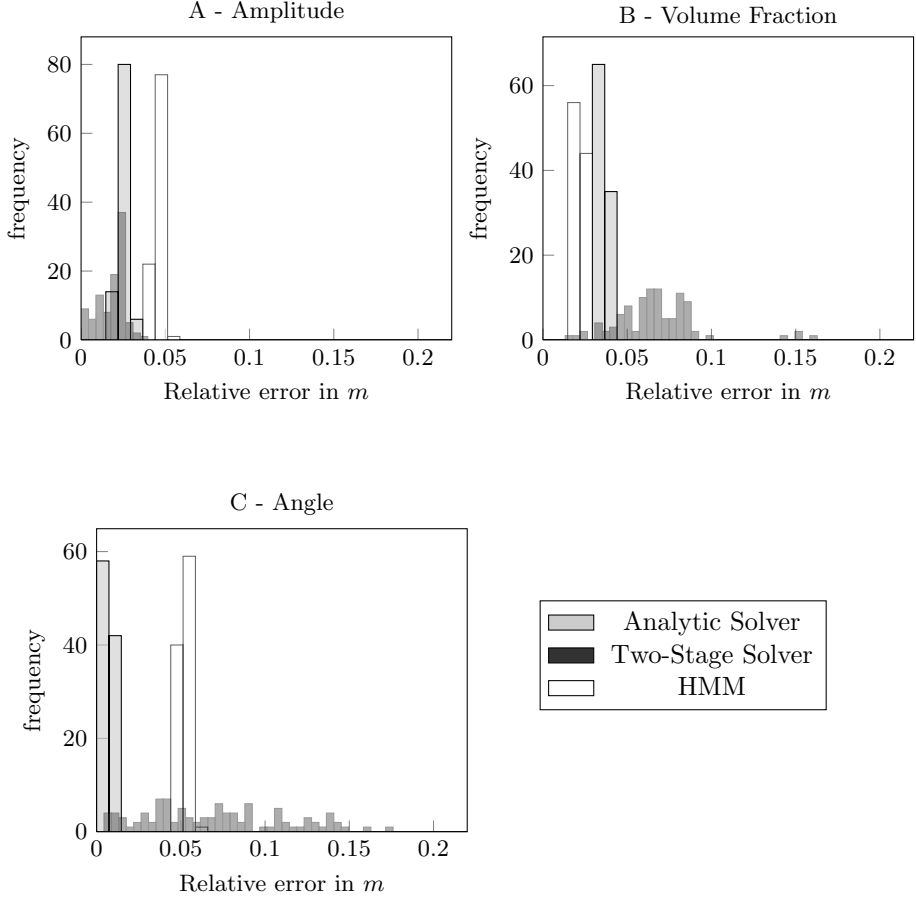


\subsection{Medical imaging ($\beps>0$)}\label{sec:qpat}
We will consider a medical imaging technique that uses a combination of optical and ultrasonic waves to determine properties of a medium from surface measurements. In quantitative Photoacoustic Tomography, (qPAT), optical coefficients are reconstructed from knowledge of the absorbed radiation map \cite{Bal2011, Bal2012}. 

Let $\Omega\subset\RR^{2}$ represent a medium of interest and $\Lambda\subset \RR_{+}$ a set of wavelengths included in the experiment. The density of photons at wavelength $\lambda$, denoted by $u(x,\lambda)$, solves the second-order elliptic equation
\begin{align}
\begin{cases} -\nabla \cdot \(a(x, \lambda)\nabla u(x,\lambda)\) + \sigma(x, \lambda)u(x,\lambda)= 0 & x\in \Omega \\
u(x,\lambda) = g(x,\lambda) & x\in \partial \Omega.
\end{cases}\labeleq{qpat}
\end{align}

Here, $a$ and $\sigma$ are diffusion and absorption coefficients  that are dependent on the wavelength $\lambda$. The ultrasound generated by the absorbed radiation is quantified by the Gr\"uneisen coefficient, $\Gamma(x)$. The objective of qPAT is to recover $(a, \sigma, \Gamma)$ using the measured data from photoacoustic experiments corresponding to an illumination pattern $g(x,\lambda)$.

We will modify the numerical examples from \cite{Bal2012} by considering the forward model \refeq{qpat} with diffusion coefficients that have variations on multiple spatial scales, $a=\aeps$. For simplicity, we will assume that the absorption and diffusion coefficients can be expressed as
\begin{align*}
\sigma(x,\lambda) =\sum_{i=1}^{2}\beta_i(\lambda)\sigma_{i}(x),\qquad \aeps(x,\lambda)=\alpha(\lambda)\aeps(x).
\end{align*}
The measured data takes the form
\begin{align}
y = \obs(a, u(\cdot, \lambda_k))_{j} = \Gamma(x_j) \sigma(x_j,\lambda_k) u(x_j, \lambda_k), \labeleq{qpatdata}
\end{align}
where the set $\{x_j\}$ consists of points in the interior of the domain $\Omega$.

\begin{figure}
\begin{center}
\includegraphics[width=.45\linewidth]{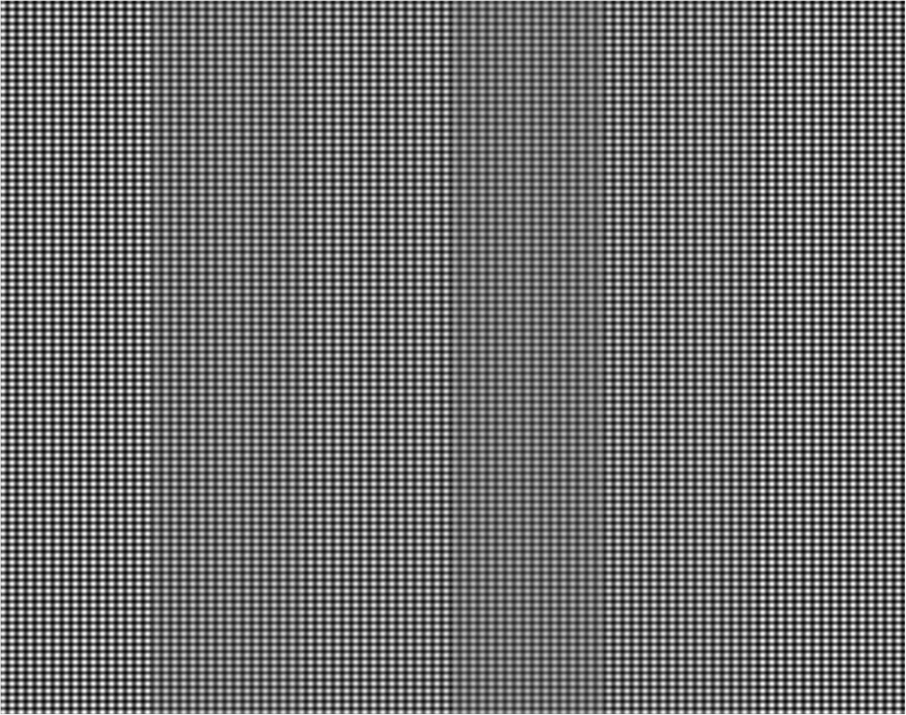} \hfill \includegraphics[width=.45\linewidth]{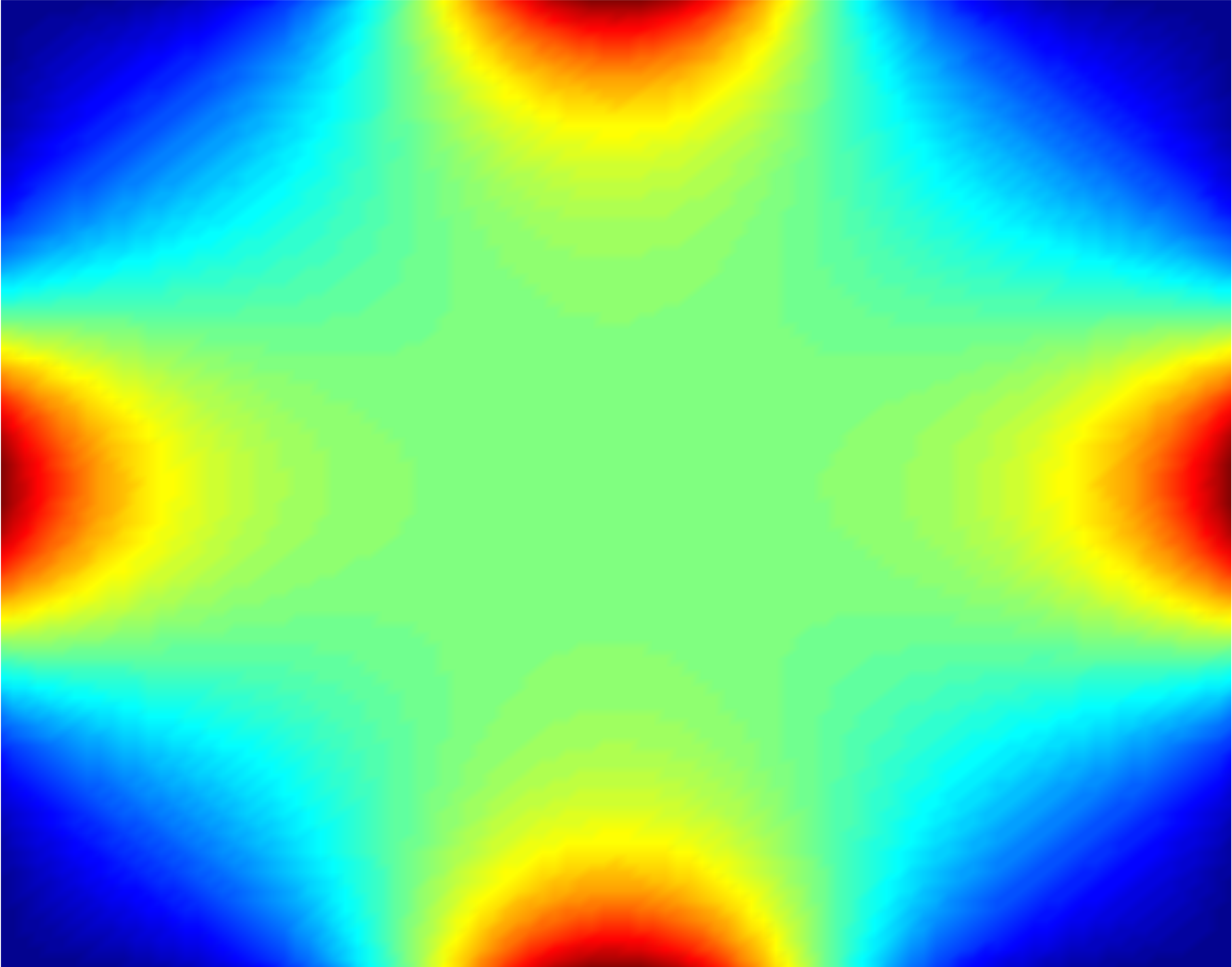}
\end{center}
\caption{Left: Multiscale coefficient $\aeps_{D}$. Right: solution to qPAT model \refeq{qpat} corresponding to four illuminations on the boundary.}
\end{figure}
In the numerical experiments, the measured data \refeq{qpatdata} involves the solutions to \refeq{qpat} for each wavelength in the set $\Lambda =\{.2, .3, .4\}$. Four illuminations are used for each wavelength.  The wavelength dependent components of the coefficients are set to be
\begin{align*}
\beta_{1}(\lambda) = \frac{\lambda}{\lambda_{0}},\quad \beta_{2}(\lambda) = \frac{\lambda_{0}}{\lambda}, \quad \alpha(\lambda)= \({\lambda}/{\lambda_{0}}\)^{3/2},
\end{align*}
where the wavelength $\lambda_{0}=.3$ normalizes the amplitude of the coefficients. The spatial components of the coefficients are given as
\begin{align*}
\Gamma(x)&=.8 + .4\tanh(4x-4),\\
\sigma_{1}(x)&= .2-.1e^{-2\pi |x-x_{0}|^{2}}, \qquad &\sigma_{2}(x)= .2+.1e^{-2\pi |x-x_{0}|^{2}}, \qquad x_{0}=(.5,.5).
\end{align*}
Figure \ref{fig:qpatcoeff} contains plots of the spatial components of  $\Gamma$ and $\sigma$.

\begin{figure}
{\includegraphics[width=.33\linewidth]{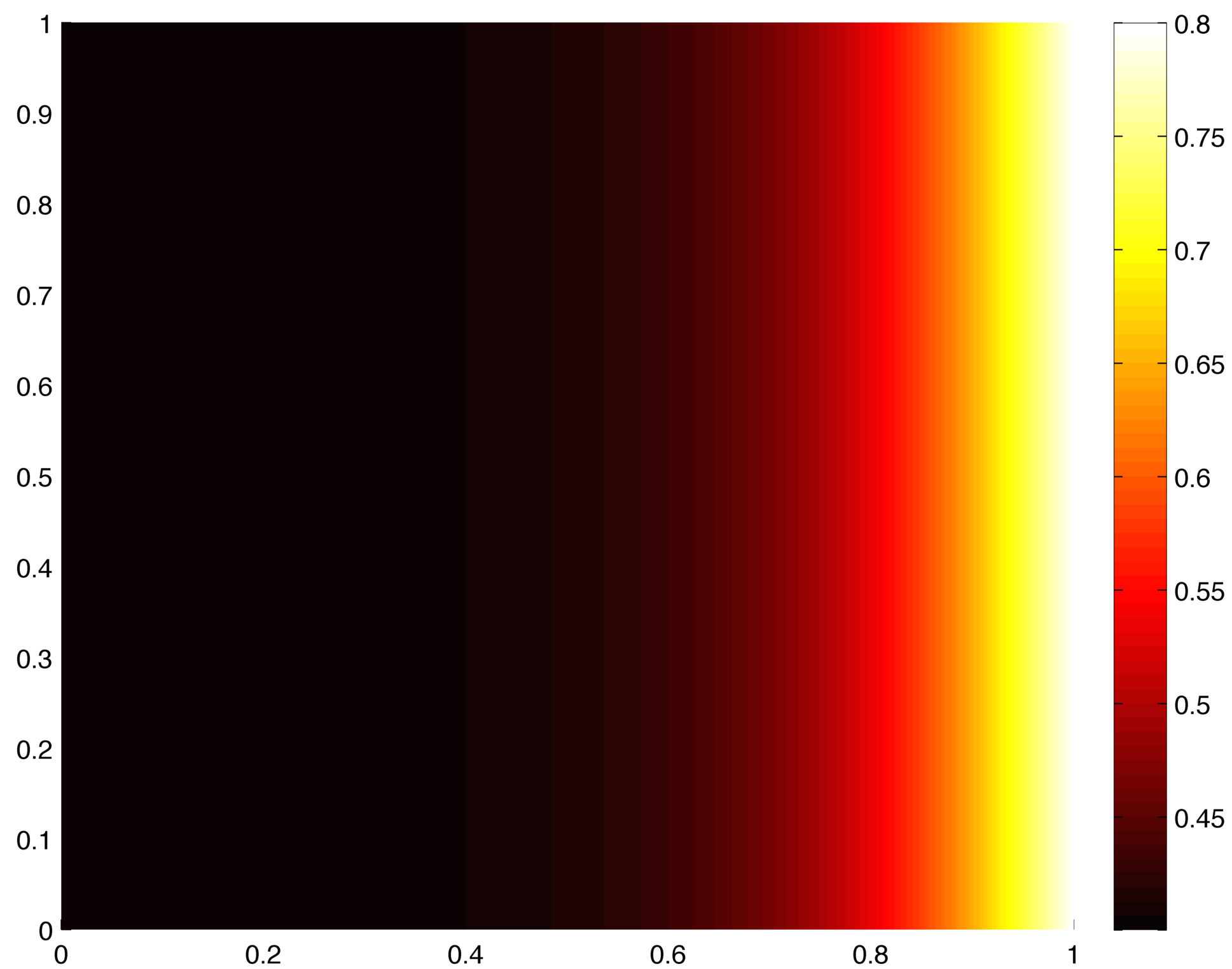}}\hfill
{\includegraphics[width=.33\linewidth]{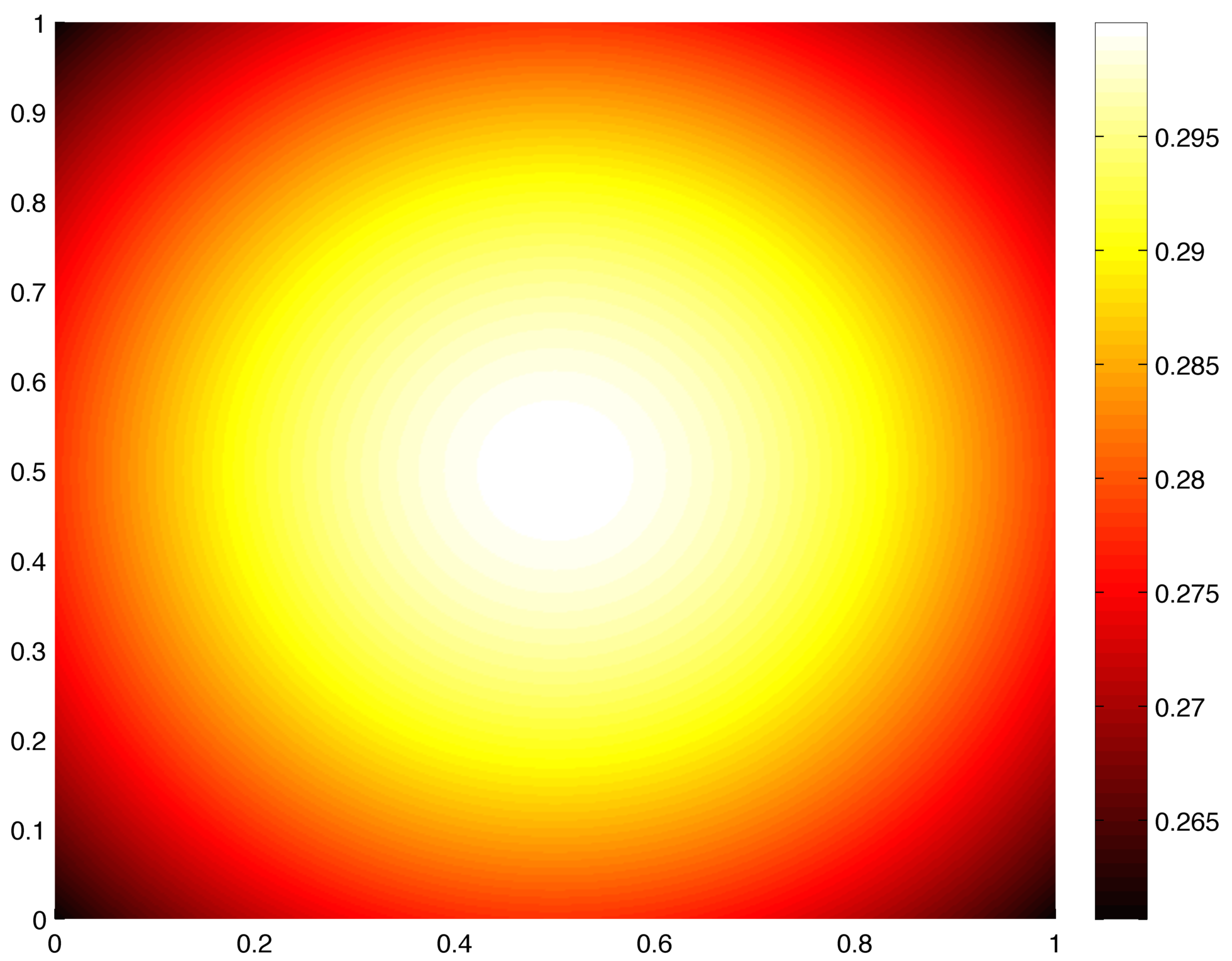}}\hfill {\includegraphics[width=.33\linewidth]{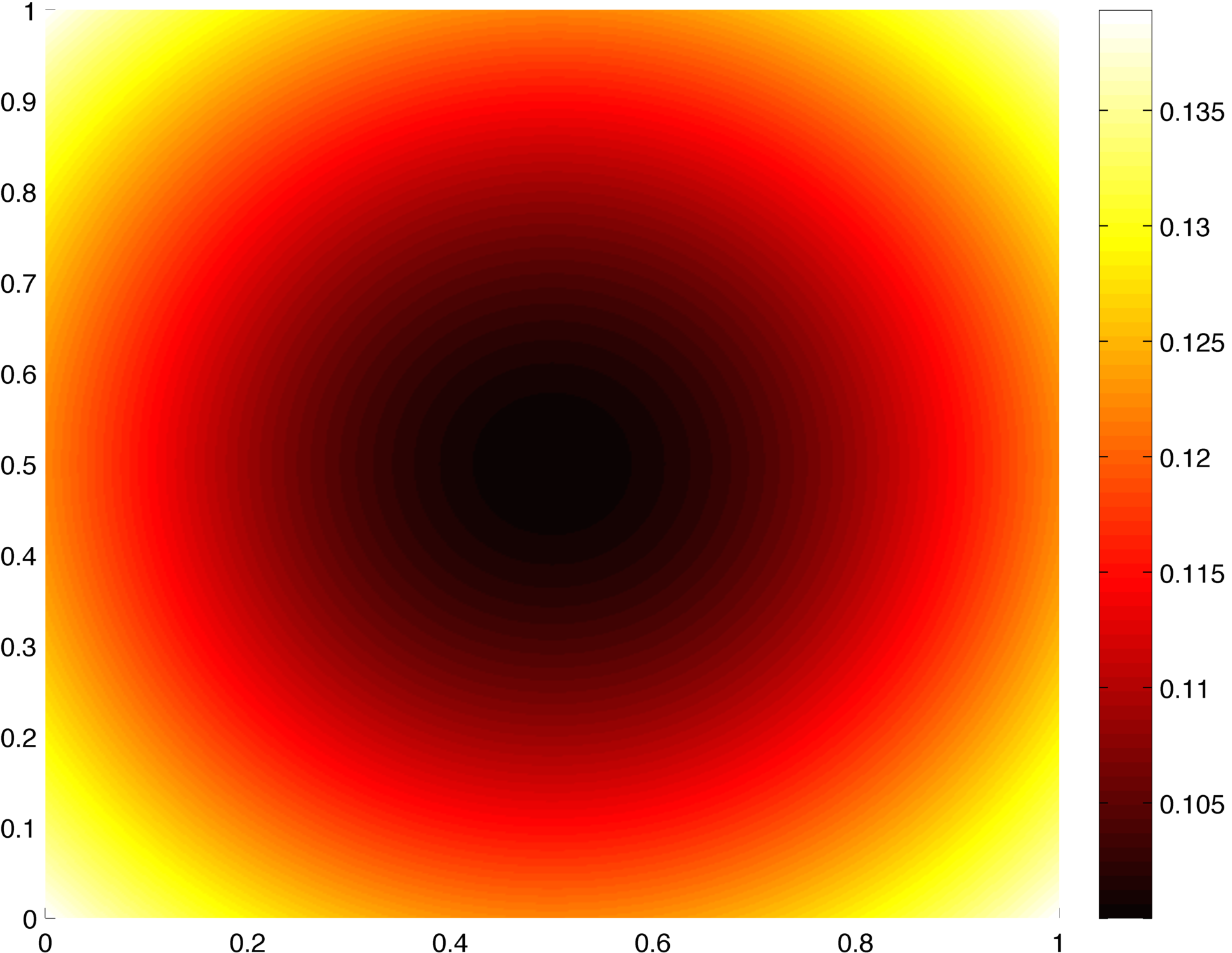}}
\caption{ {\bf Spatial components of the qPAT model.} Left to right: Gr\"uneisen coefficient $\Gamma(x)$,  and absorption component functions $\sigma_{1}(x)$ and $\sigma_{2}(x)$.}
  \label{fig:qpatcoeff}
\end{figure}

The microstructure in models \ref{AMP2d} and \ref{VF2D}, respectively, are represented by the coefficients
\begin{align}
\aeps_{D}(m(x),x) &= \aeps_{A}(m(x_{1}),x_{1})\aeps_{A}(m(x_{2}),x_{2}), \\
\aeps_{E}(m(x),x) &= .5+2\chi_{\{[0, m]\times [0, m]\}}(x). \\
\end{align}

As in \cite{Bal2012}, the reconstruction errors are given for synthetic data with no noise added. The results in Table \ref{table:qpat} show inversion results using a HMM forward solver for macroscopic predictions. Here the parameters are chosen to be $\epsilon=1/100$, $H=1/20$, $\delta = 3\epsilon$, and $h=1/800$. The microscale parameter $m(x)$ is a piecewise constant function of the form \refeq{mpw}.
\begin{table}[h]
\caption{Inversion errors for parameter estimation in qPAT models containing periodic cell structures.}
\begin{center}
\begin{tabular}{|c|c|c|}
\hline
$N$ & \ref{AMP2d} - Amplitude & \ref{VF2D} - Volume Fraction \\
\hline
1		&  0.05605865 &   0.01529364 \\
2		& 0.07461357  &  0.01838472 	 \\
3		&  0.07146362 &   0.02986680 	 \\
4		&  0.08319615   &   0.03748564 	 \\
5		& 0.13636728 &   	0.04616734 \\
6		&  0.08914877 &   	0.03729032  \\
\hline
\end{tabular}
\end{center}
\label{table:qpat}
\end{table}%


\subsection{Seismic waveform inversion ($\beps<0$)}\label{sec:helm}

In exploration geophysics, scientists attempt to determine the geological properties of the Earth's crust that govern the propagation of acoustic waves (see \cite{Symes2009} for an overview). In full waveform inversion, the goal is to find a subsurface model that produces the best fit to reflection data recorded from seismic surveys. Each prediction is simulated using the physics of the experiment. This corresponds to an inverse problem for partial differential equations where the unknown coefficients represent properties of the sedimentary layers, e.g. velocities, porosity, and saturation. 

Full waveform inversion is the result of combining numerical methods for the simulation of wave propagation with optimization techniques to minimize the data misfit term (see \cite{Fichtner2013} for a discussion of multiscale full waveform inversion). Traditional finite element methods (FEM) or finite difference methods (FDM) for wave propagation in the high frequency regime come with a considerably high computational cost due to the highly oscillatory nature of the propagating waves  \cite{Ihlenburg1997}. 

The forward problem can be modeled in both the time domain and the frequency domain. In theory, both approaches are equivalent, however the choice of model can influence the design of specific numerical methods to optimize performance. An advantage of the frequency domain model is that a coarse discretization of the frequencies can be used to produce images that are free from aliasing \cite{Baeten2013,Mulder2008,Sirgue2004}. 

A major hurdle in full waveform inversion is the presence of local minima in the least-squares functional for the data misfit. In \cite{Plessix2006}, adjoint-state methods are used to efficiently calculate the gradient of the least-squares functional and speed up the optimization. We emphasize that in this work we use standard optimization routines in order to fully study the effects of fitting an effective model to the data.

Our numerical examples correspond to problems that mimic the models used in seismic waveform inversion. Here, the model parameters represent the spatially varying volume fraction, angle, and amplitude of the layers. The forward model $G$ maps $a$ to the solution to the 2D variable coefficient Helmholtz equation on the square $\Omega = [0,1]^{2}$, 
\begin{align}
\nabla \cdot (a(x)\nabla u )+\omega^{2}u(x) &= \delta(x-x_{s})\quad x \in {\Omega},\labeleq{helm}
\end{align}
where $a$ is the model parameter that characterizes the density of the medium, $\omega$ is the wave number, and $u$ is the spatially varying pressure field arising from a disturbance at a source located at $x_{s}\in \Omega$. We impose the absorbing boundary condition
\begin{align}
a \nabla u \cdot n  - ik u  = 0 \quad \text{on $\partial \Omega$},
\end{align}
where $k=a^{-1/2}\omega$.  The seismic data is represented as the collections of solutions measured on the sensor domain ${{D}}\subset\Omega$,  $\obs(a, u_k)_{j} = u_k(x_j)$, $\{x_j\}\subset{{D}}$ (see Figure \ref{fig:helmdns}).

\begin{figure}
{\includegraphics[width=.25\linewidth]{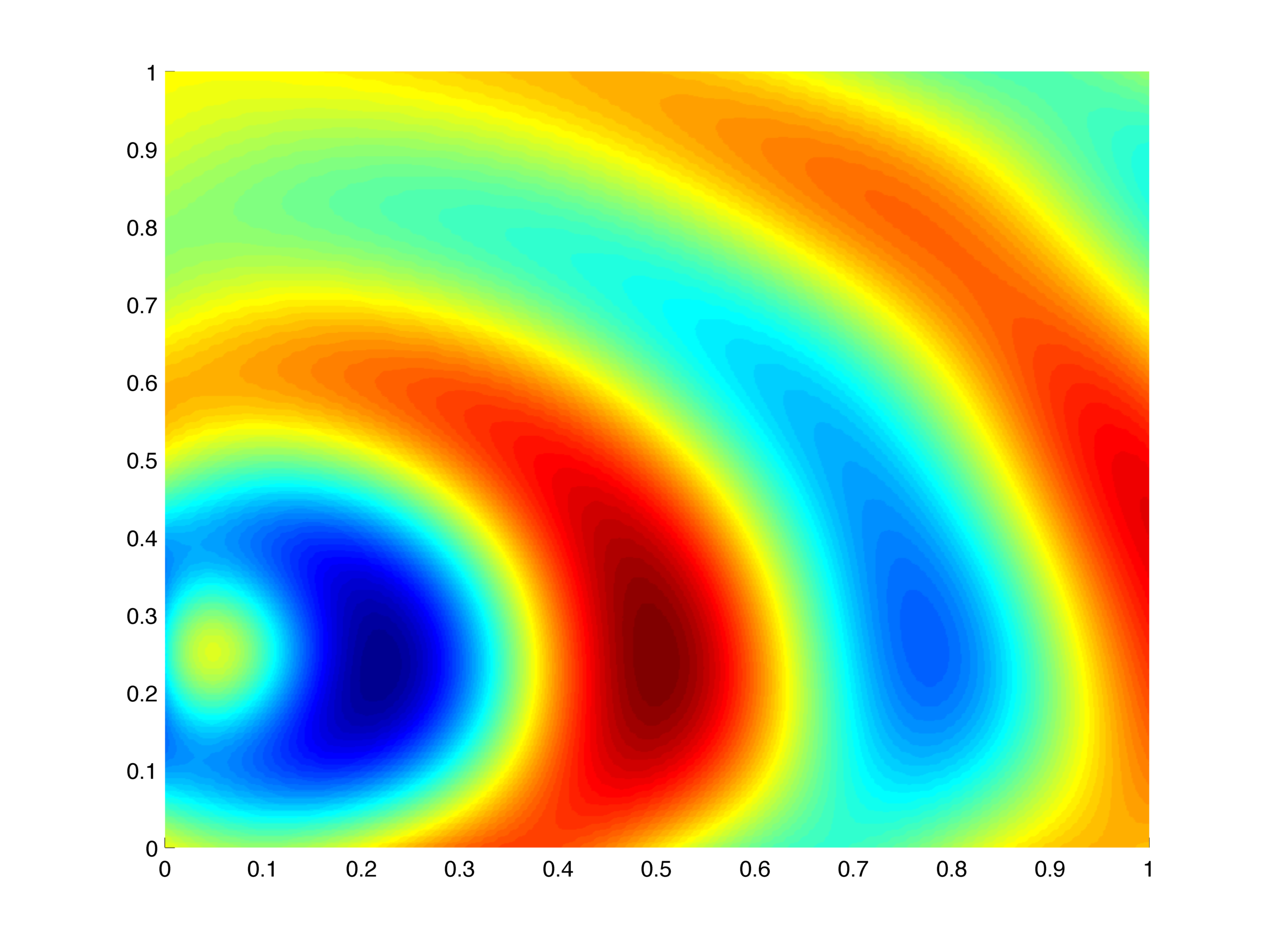}\hfill \includegraphics[width=.25\linewidth]{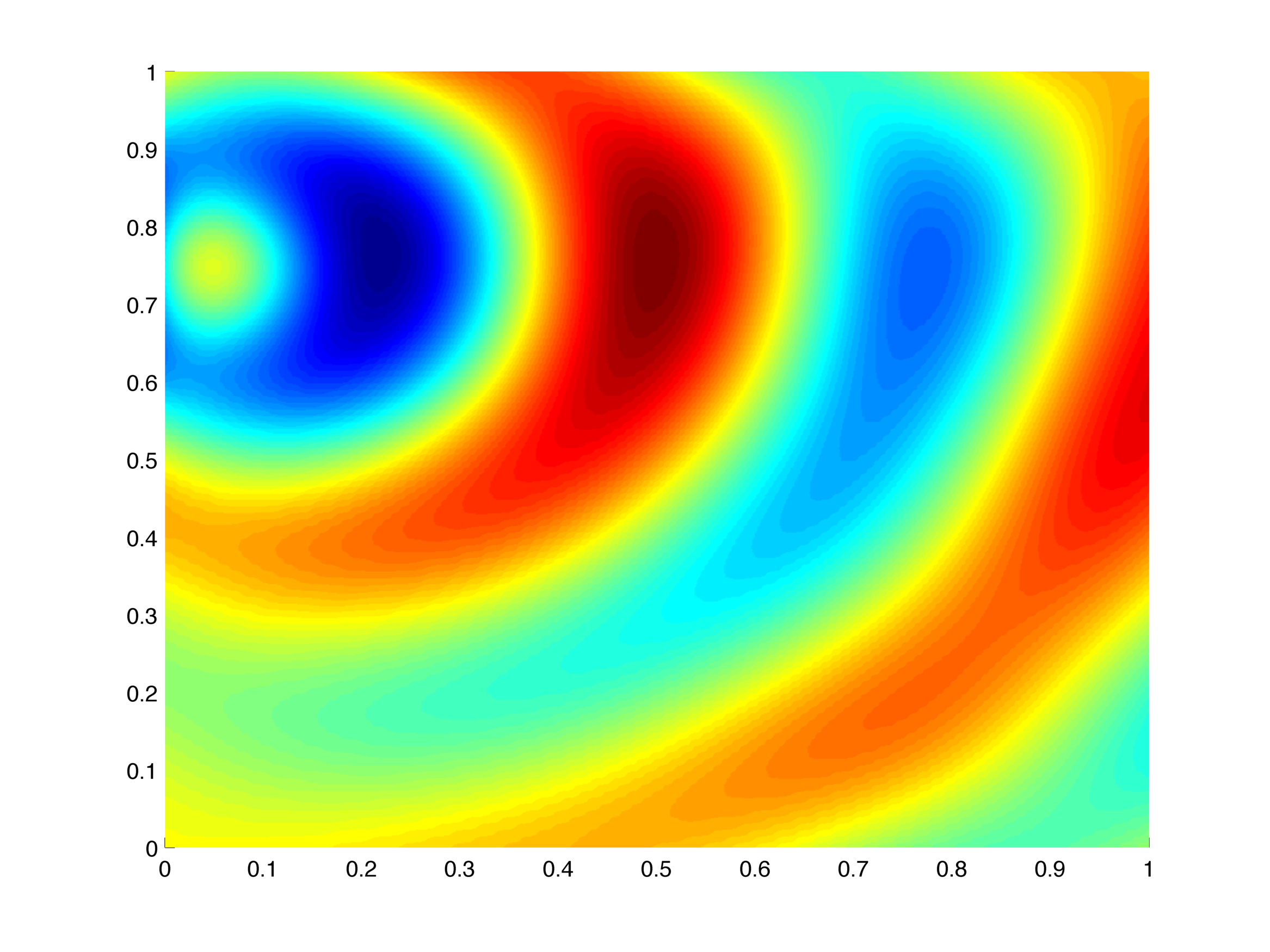}\hfill \includegraphics[width=.25\linewidth]{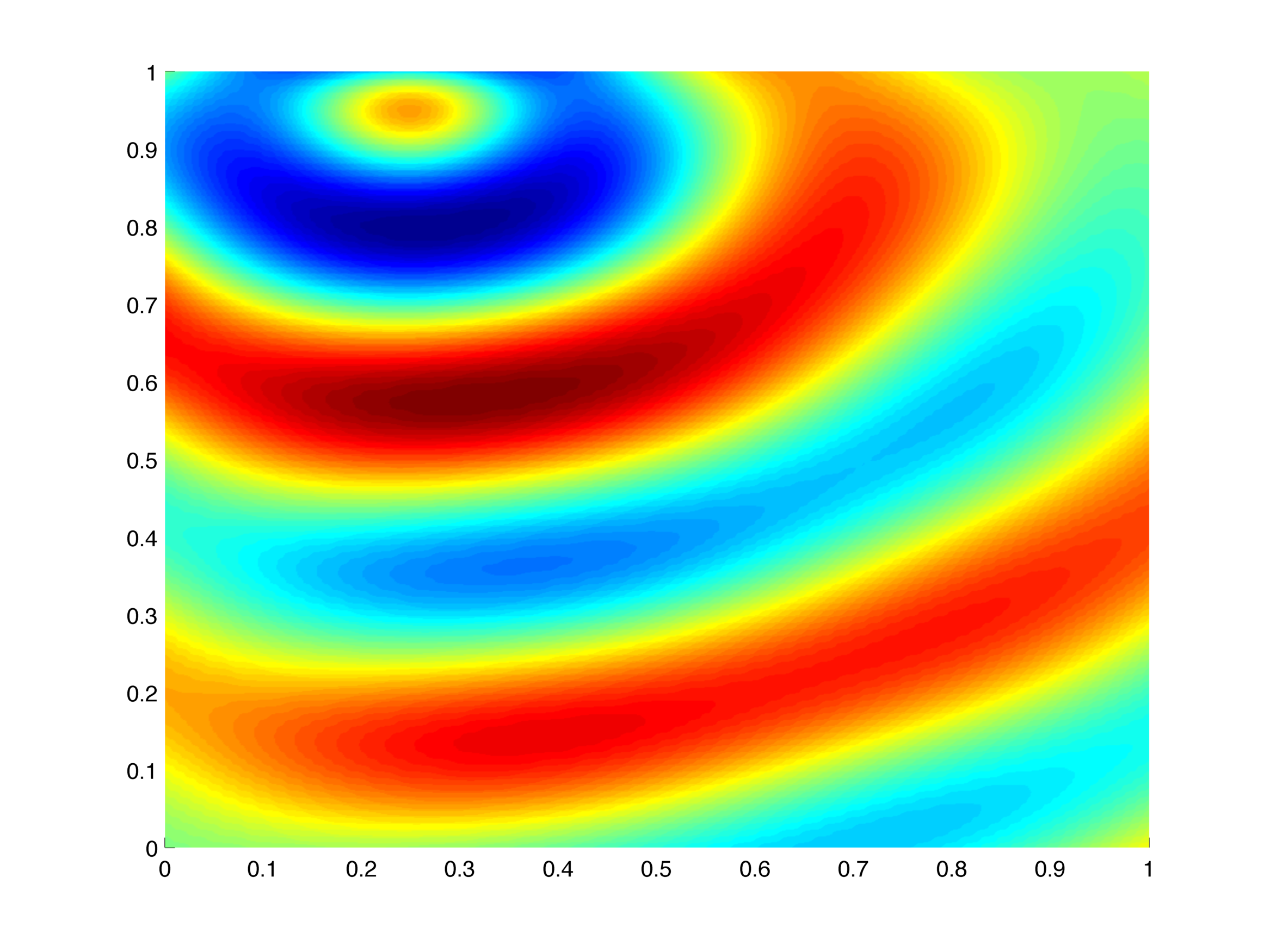}\includegraphics[width=.25\linewidth]{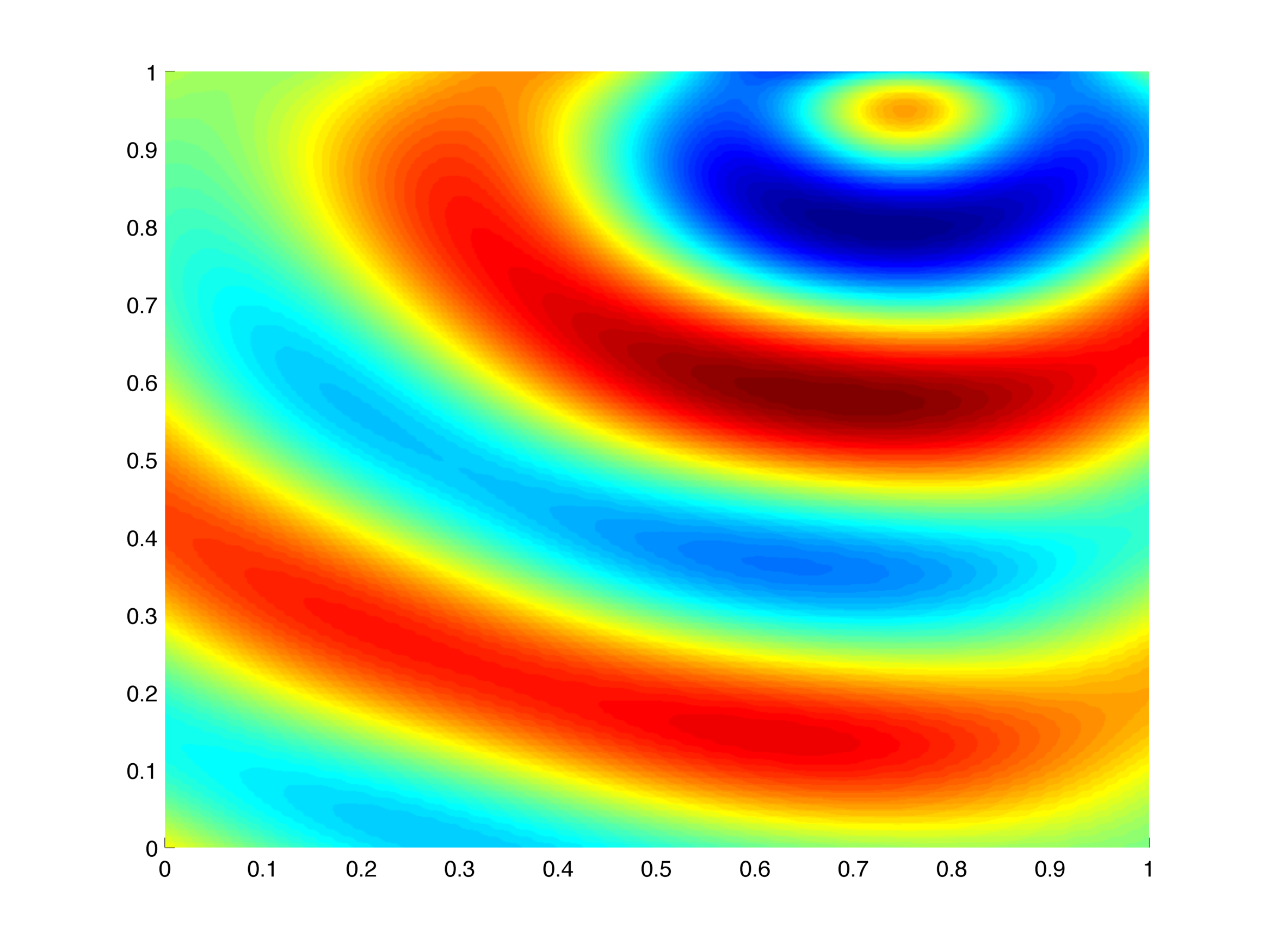}}
\caption{Solutions of the multiscale Helmholtz equation \refeq{helm} for wavelength $\omega = 4\pi$  and multiscale coefficient $a= \aeps_{A}$.}
\label{fig:helmdns}
\end{figure}

In the numerical simulations,  we set $\omega=4\pi$ and the Dirichlet data is obtained from solutions corresponding to multiple sources at $x_{s_{1}}=(0, .25)$, $x_{s_{2}}=(0, .75)$, $x_{s_{3}}=(.25, 1)$, and $x_{s_{4}}=(.75, 1)$. Results shown in Table \ref{table:helm} demonstrate microscale inversion of the Helmholtz equation \refeq{helm} using the methods described earlier. Here the parameters are chosen to be $\epsilon=1/120$, $H=1/40$ and $h=1/800$, $\delta=6\epsilon$.
   
\begin{table}[h]
\caption{Inversion error for parameter estimation in the Helmholtz model.}
\begin{center}
\begin{tabular}{|c|c|c|c|}
\hline
$N$ & \ref{AMP} - Amplitude & \ref{VF} - Volume Fraction & \ref{ANG} - Angle \\
\hline
1		&  0.02776429 &   0.03736824  &0.03891544 	 \\
2		&  0.04400273   &  0.02612802 	 &0.07489835\\
3		&  0.04607522  &  0.01553997	 &0.06326932\\
4		&   0.07915712  &   0.00976197	 & 0.25927552\\
5		&  0.05350197 &   	 0.01566102&0.17280984 \\
6		&  0.04874272 &   	0.01968874  & 0.21757778\\
\hline
\end{tabular}
\end{center}
\label{table:helm}
\end{table}%


\section{Conclusion}\label{sec:conclusion}

We present computational techniques for solving inverse problems for multiscale partial differential equations. Our goal is to recover microscale information using PDE constrained optimization. Instead of directly working with the effective equation we constrain the search space by representing the microscale by a limited number of parameters in order to have a well-posed inverse problem. When a parameter based effective model exists we use that,  otherwise,  the numerical heterogeneous multiscale method (HMM) can be used even when the explicit form of the effective equation is not known. By applying recovery results for inverse conductivity problems with special anisotropy \cite{Alessandrini2001}, we can prove that certain microstructure features can be determined uniquely from the Dirichlet to Neumann map corresponding to the effective equations. We provide numerical examples, which show quantitative convergence information with respect to numerical resolution, scale separation and parameterization strategies. We also provide numerical results that demonstrate the performance of these techniques applied to random media and simple models with lower order terms of the form used in medical imaging and exploration seismology.

The goal of the current research has been a proof of concept and there are natural future directions outside the scope of the current paper. For example, in more realistic applications where higher resolution is required, other minimization techniques must be used. Good candidates would be adjoint-state based methods, which are used in full waveform inversion \cite{Plessix2006}. Another direction is to further probe random cases and explore the use of multiple parameters in connection to known prior information in specific applications.

\section*{Acknowledgements} This work has benefited from valuable discussions with Assyr Abdulle, Kui Ren, Pingbing Ming and Fenyang Tang. This research was supported in part by NSF grant DMS-1217203, the Texas Consortium for Computational Seismology, and Institut Mittag-Leffler. CF was also supported in part by NSF grant DMS-1317015.

\bibliographystyle{SIAM}

\end{document}